\documentclass[a4paper,11pt,reqno]{amsart}
\usepackage[english]{babel}
\usepackage[utf8]{inputenc}
\usepackage[T1]{fontenc}
\usepackage{amsthm,amsmath,amssymb,amsrefs,enumerate,enumitem,mathtools,mathrsfs,lmodern,microtype}
\usepackage[a4paper, left=2.5cm, right=2.5cm]{geometry}
\usepackage[colorlinks=true, linkcolor=blue, citecolor=red, urlcolor=blue]{hyperref}
\usepackage{fancyhdr}
\DeclareMathOperator{\Var}{Var}
\newtheorem{thm}{Theorem}[section]
\newtheorem{prop}[thm]{Proposition}
\newtheorem{lemma}[thm]{Lemma}

\theoremstyle{definition}
\newtheorem{defn}[thm]{Definition}
\newtheorem{assumption}[thm]{Assumption}
\theoremstyle{remark}
\newtheorem{rmk}[thm]{Remark}

\newcommand{\R}{\mathbb{R}}
\newcommand{\E}{\mathbb{E}}
\newcommand{\al}{\alpha}
\newcommand{\cH}{\mathcal{H}}
\newcommand{\cL}{\mathcal{L}}
\newcommand{\cS}{\mathcal{S}}

\newcommand{\eps}{\varepsilon}
\newcommand{\cF}{\mathcal{F}}
\newcommand{\Pb}{\mathbb{P}}
\newcommand{\Ga}{\Gamma}
\newcommand{\cE}{\mathcal{E}}
\newcommand{\la}{\lambda}

\usepackage{xcolor}
\definecolor{revisionblue}{RGB}{0,70,170}
\definecolor{revred}{RGB}{225,0,0}
\definecolor{revgreen}{RGB}{0,140,0}

\title[Fractional Navier-Stokes Equations Driven by Hermite Noise]
{Fractional Navier-Stokes Equations with Caputo Derivative Driven by Hermite Noise}

\author{Lara Daw} 
\address{Department of Mathematics and Statistics, Zayed University} 
\email{lara.daw@zu.ac.ae}

\author{Atef Lechiheb}
\address{D\'epartement de math\'ematique, Toulouse School of Economics, Universit\'e de Toulouse Capitole, Toulouse, France}
\email{atef.lechiheb@tse-fr.eu}

\keywords{ Stochastic Navier--Stokes equations; Caputo derivative; Hermite noise; Maximal mild solutions; H\"older regularity; Functional non-central limit theorem }

\date{\today}

\begin{document}

\begin{abstract}
We study time-fractional stochastic Navier-Stokes equations on a bounded domain of $\R^2$, driven by a cylindrical Hermite noise of order $k\ge1$ and Hurst parameter $H\in(1/2,1)$, built mode-by-mode on the eigenbasis of the Stokes operator. This class of noises generalizes fractional Brownian motion ($k=1$) and the Rosenblatt process ($k=2$) to arbitrary non-Gaussian order. Combining Wiener-chaos hypercontractivity with Hilbert-Schmidt estimates for the Mittag-Leffler operator, we establish the regularity of the associated stochastic convolution. A Da Prato-Debussche decomposition, carried out in a scale of Lebesgue spaces $L^q_\sigma(D)$ on which the bilinear term is controlled by a Calder\'on-Zygmund estimate, then yields existence and uniqueness of a maximal mild solution, and H\"older regularity of the time-frozen solution under the condition $\al+2H>2$ together with standard regularity assumptions on the initial datum.  We further obtain moment bounds for the stochastic convolution and for the quadratic self-interaction of the noise, whose growth in the moment order is governed by $k$, and a functional non-central limit theorem giving a radially localised approximation of the solution by systems driven by functionals of Gaussian sequences with long-range dependence.
\end{abstract}

\maketitle

\tableofcontents

\bigskip

\section{Introduction}

\subsection{Overall motivation}

The stochastic Navier-Stokes equations are a fundamental model for turbulent fluid flows. Since the work of Bensoussan and Temam \cite{bensoussan1973equations}, these equations have been extensively studied, first with Gaussian white noise \cite{da2002two, mikulevicius2005global}, and more recently with fractional Brownian motion \cite{zou2018stochastic}. The deterministic Navier-Stokes equations describe the motion of viscous incompressible fluids. Adding a random forcing allows one to model turbulence, thermal fluctuations, and external noise. For ergodic properties, see \cite{flandoli1995ergodicity, hairer2006ergodicity}. For related results on stochastic Burgers equation with Hermite noise, see \cite{lechiheb2026burgers}.

In many physical systems, memory effects play an important role. Complex fluids, viscoelastic materials, and flows in porous media exhibit anomalous diffusion that cannot be captured by ordinary derivatives. Fractional calculus provides a natural framework for such phenomena \cite{podlubny1999fractional, kilbas2006theory}. The Caputo derivative is particularly convenient because it allows initial conditions with a clear physical interpretation. Zou, Lv and Wu \cite{zou2018stochastic} combined these ideas: they studied stochastic Navier-Stokes equations with a Caputo time-fractional derivative driven by fractional Brownian motion. They proved existence, uniqueness, and H\"older regularity of mild solutions using Mittag-Leffler operators and stochastic convolution estimates. For deterministic counterparts, see \cite{carvalho2015mild, zhou2017time, zhou2017weak}.

The present paper extends their work by replacing fractional Brownian motion with a Hermite process $Z_H^k$. For an integer $k\ge1$ and Hurst parameter $H\in(1/2,1)$, the Hermite process is defined as a $k$-fold Wiener-It\^o integral. It shares the covariance of fractional Brownian motion,
\begin{equation}
\label{eq:hermite_cov}
\E[Z_H^k(t)Z_H^k(s)] = \frac12\bigl(t^{2H}+s^{2H}-|t-s|^{2H}\bigr),
\end{equation}
but is non-Gaussian for $k\ge2$. For $k=1$ we recover fractional Brownian motion; for $k=2$ we obtain the Rosenblatt process \cite{tudor2008analysis}, which appears naturally in non-central limit theorems for quadratic functionals of Gaussian sequences \cite{dobrushin1979non, taqqu1979convergence}. The stochastic calculus for Hermite processes was developed by Maejima and Tudor \cite{maejima2007wiener}, who defined Wiener integrals and proved a non-central limit theorem. For related SPDEs with Hermite noise, see \cite{lechiheb2026burgers, slaouitudor2019}.

Beyond the long-range dependence captured by $H$, the Hermite order $k$ is itself a physically meaningful parameter, independent of $H$: replacing fractional Brownian motion by a higher-order Hermite process leaves the second-order (covariance) structure of the forcing unchanged, but introduces increasingly heavy-tailed, skewed fluctuations, with excess kurtosis growing in $k$ \cite{tudor2023nongaussian}.    This decoupling of memory (governed by \(H\)) from tail behavior (governed by \(k\)) may be useful in models of turbulent flows in which long-range temporal correlations coexist with strongly non-Gaussian fluctuations.

\subsection{Description of the model}

We consider the following time-fractional stochastic Navier-Stokes equations on a bounded domain $D\subset\R^2$ with Dirichlet boundary conditions:
\begin{equation}
\label{eq:main}
\begin{cases}
{}^C D_t^\al u = -A u + B(u) + f(u) + \dot{Z}_H^k(t), \quad t>0,\\
u(0) = u_0.
\end{cases}
\end{equation}
Here ${}^C D_t^\al$ is the Caputo derivative of order $\al\in(0,1)$, $A$ is the Stokes operator, with eigenbasis $\{e_j,\lambda_j\}_{j\ge1}$ (Section~\ref{sec:prelim}), $B(u)=B(u,u)$ with $B(u,v)=-P_H[(u\cdot\nabla)v]$ is the bilinear convective term, and $f$ is a Lipschitz continuous external force. The driving noise $\dot Z_H^k$ is the formal time derivative of a \emph{cylindrical} Hermite process, built mode-by-mode on the Stokes eigenbasis:
\begin{equation}
\label{eq:cylindrical-noise}
Z_H^k(t) := \sum_{j\ge1} Z_{H,j}^{(k)}(t)\, e_j,
\end{equation}
where $\{Z_{H,j}^{(k)}\}_{j\ge1}$ are independent, identically distributed copies of a scalar Hermite process of order $k$ and Hurst parameter $H$ (Definition~\ref{def:hermite-cylindrical}, Section~\ref{subsec:hermite}), each built from an independent underlying Brownian motion. As for a cylindrical Wiener process, the series \eqref{eq:cylindrical-noise} does not converge in $L^2_\sigma(D)$; it is meaningful only through the stochastic convolution it drives, which we define mode-by-mode, precisely, in Section~\ref{sec:convolution}.

The mild formulation of \eqref{eq:main} is
\begin{equation}
\label{eq:mild}
u(t) = E_\al(t) u_0 + \int_0^t \cS_\al(t-s) \bigl[B(u(s)) + f(u(s))\bigr] \, ds + Z_k(t),
\end{equation}
where $E_\al$ and $\cS_\al(t)=t^{\al-1}E_{\al,\al}(t)$ are Mittag-Leffler operators (Section~\ref{sec:prelim}), and
\begin{equation}
\label{eq:conv-def}
Z_k(t) := \sum_{j\ge1}\left(\int_0^t s_{\al,j}(t-s)\,dZ_{H,j}^{(k)}(s)\right) e_j,
\qquad s_{\al,j}(r):=r^{\al-1}E_{\al,\al}(-\lambda_j r^\al),
\end{equation}
is the stochastic convolution, defined mode-by-mode consistently with \eqref{eq:cylindrical-noise} (the eigenvalue $s_{\al,j}(t-s)$ being simply the scalar action of $\cS_\al(t-s)$ on the eigenvector $e_j$).

Three main difficulties arise when analyzing \eqref{eq:mild}. First, the stochastic integral defining $Z_k$ must be shown to converge in an appropriate space and its $L^p(\Omega)$ moments controlled; for $k\ge2$ the noise is non-Gaussian, so classical Gaussian isometries are unavailable and are replaced throughout by Wiener-chaos hypercontractivity (Section~\ref{subsec:hermite}). Second, the operator $\cS_\al$ has an intrinsic singularity at $t=0$, so its smoothing and increment properties, in Hilbert-Schmidt norm, must be quantified with care: no increment bound for $\cS_\al(t_2-\cdot)-\cS_\al(t_1-\cdot)$ holds uniformly down to the coincidence point, and a near/far decomposition of the time integral is required to reach the sharp regularity of $Z_k$ (Section~\ref{sec:convolution}). Third, the bilinear term $B(u)$ is too singular in the noise variable to admit a direct fixed-point argument for $u$ itself: we instead separate $u=v+Z_k$ into a smoother remainder $v$ and the rougher stochastic convolution, following a Da Prato-Debussche-type decomposition, and work in a scale of Lebesgue spaces $L^q_\sigma(D)$ on which the bilinear term is controlled by a Calder\'on-Zygmund estimate (Section~\ref{sec:existence}).

\subsection{Main results}

We now state the main results of the paper: well-posedness, regularity, moment estimates with explicit dependence on the moment order and on the Hermite order $k$, and a non-central limit theorem.

\begin{thm}[$L^p$ estimates for the Hermite integral]\label{thm:hyper}
For any deterministic function $f$ in the covariance Hilbert space $\cH$ of $Z_H^k$ and any $p\ge2$,
\[
\Big\|\int_{\R} f(u)\,dZ_H^k(u)\Big\|_{L^p(\Omega)} \le (p-1)^{k/2} \|f\|_{\cH}.
\]
\end{thm}

This estimate holds for the scalar Wiener integral against a single mode of the cylindrical noise \eqref{eq:cylindrical-noise}; it is the building block for the mode-by-mode construction of the vector-valued stochastic convolution $Z_k$ carried out in Section~\ref{sec:convolution}.

\begin{thm}[Regularity of the stochastic convolution]\label{thm:Zk}
Assume $0<\alpha<1$, $H\in(1/2,1)$, $0\le\nu<1$, and $\alpha(1-\nu)+2H>2$. Then for any $p\ge2$ and $t\in[0,T]$,
\[
\|Z_k(t)\|_{L^p(\Omega;\dot{H}^\nu)} \le C_{p,k}\, t^{\frac{\alpha(1-\nu)+2H-2}{2}},
\qquad C_{p,k}=(p-1)^{k/2}C_0.
\]
Moreover, for $0\le t_1<t_2\le T$,
\[
\|Z_k(t_2)-Z_k(t_1)\|_{L^p(\Omega;\dot{H}^\nu)} \le C_{p,k}\, (t_2-t_1)^{\gamma},
\qquad \gamma = \frac{\alpha(1-\nu)+2H-2}{2}.
\]
\end{thm}

The value and increment estimates share the same exponent $\gamma$. This is not automatic: the operator $\cS_\al$ itself admits no difference bound $\|[\cS_\al(t_2)-\cS_\al(t_1)]\chi\|_\nu\le C(t_2-t_1)^\theta\|\chi\|$ valid uniformly down to $t_1=0$ (Proposition~\ref{prop:endpoint-fact}), so one might expect the increment of $Z_k$ to be governed by a second, smaller exponent inherited from that singularity. Section~\ref{sec:convolution} shows that it is not: splitting the time integral into the ranges near and far from the coincidence point, both contributions scale at exactly the rate $\gamma$.

\begin{thm}[Existence and uniqueness of mild solutions]\label{thm:exist}
Assume
\[
\al+2H>2,
\]
and let $q$ be any admissible Lebesgue exponent satisfying $2<q<\al/(1-H)$ (a nonempty range precisely under the above condition). Let $u_0\in L^q_\sigma(D)$.  Then there exist a random time $\tau^*\in(0,\infty]$, almost surely positive, and a unique maximal mild solution $u\in C([0,\tau^*);L^q_\sigma)$ of \eqref{eq:mild}; moreover, on $\{\tau^*<\infty\}$, $\limsup_{t\uparrow\tau^*}\|u(t)-Z_k(t)\|_{L^q_\sigma}=\infty$. (We write $T^*\le\tau^*$ for the local existence time of Section~\ref{sec:existence}.)
\end{thm}

The solution is local in time as constructed in Section~\ref{sec:existence}, but it can be continued up to a maximal time $\tau^*$ by the history-preserving argument of Section~\ref{sec:existence}; no restriction on the parameters beyond $\al+2H>2$, $2<q<\al/(1-H)$ is needed.

Rather than seeking a fixed point for $u$ directly, the proof separates the stochastic convolution -- whose regularity is limited to $\dot{H}^{\nu_q}$, $\nu_q=1-2/q$, by Theorem~\ref{thm:Zk}, and which never enters $\dot{H}^{\nu_q+1}$ -- from a smoother remainder $v:=u-Z_k$, solving
\[
v(t) = E_\al(t)u_0 + \int_0^t \cS_\al(t-s)\bigl[B(v(s)+Z_k(s))+f(v(s)+Z_k(s))\bigr]\,ds,
\]
for which a Banach fixed-point argument closes directly in $C([0,T^*];L^q_\sigma)$, via a Calder\'on-Zygmund estimate for the bilinear term that requires no gain of Sobolev regularity beyond $L^q_\sigma$ itself; the randomness of $T^*$ reflects the fact that the ball radius entering the contraction argument depends on the (random) size of the noise term $Z_k$, not on any deficiency of the method. This is carried out in Section~\ref{sec:existence}.

\begin{thm}[H\"older regularity]\label{thm:holder}
Under the assumptions of Theorem~\ref{thm:exist}, suppose in addition $u_0\in D(A_q^\theta)$, the domain of the Stokes operator in $L^q_\sigma$, for some $\theta\in(0,1]$. Then for every $T>0$, $p\ge2$ and $0\le t_1<t_2\le T$,
\[
\Bigl(\E\bigl[\|u(t_2\wedge T^*)-u(t_1\wedge T^*)\|_{L^q_\sigma}^p\bigr]\Bigr)^{1/p}\le C\,(t_2-t_1)^{\beta},
\qquad
\beta=\min\Bigl\{\al\theta,\ \al\bigl(\tfrac12-\tfrac1q\bigr),\ \tfrac{\al}{q}+H-1\Bigr\}.
\]
\end{thm}

 The three exponents arise respectively from the regularity of the initial condition, the smoothing properties of the Mittag--Leffler operators applied to the nonlinear term, and the Hölder regularity of the stochastic convolution.

\begin{thm}[Functional non-central limit approximation]\label{thm:nclt}
Let $0<\al<1$, $H\in(\frac12,1)$, $2<q<\al/(1-H)$, $k\ge1$, $u_0\in L^q_\sigma$, and let $f$ satisfy Assumption~\ref{ass:f}. Put $H_0:=1-\frac{1-H}{k}\in(1-\frac1{2k},1)$. For each Stokes mode $j\ge1$ let $\{\xi^{(j)}_n\}_{n\in\mathbb Z}$ be independent copies of a stationary centered Gaussian sequence with unit variance and covariance $\rho(m)=m^{-2(1-H)/k}L(m)$, $L$ slowly varying at infinity. Let $d_N^2:=\Var\bigl(\sum_{n=1}^NH_k(\xi^{(1)}_n)\bigr)$ and define the piecewise linear interpolations $\widetilde S_N^{(j)}(t):=d_N^{-1}\bigl[\sum_{n=1}^{\lfloor Nt\rfloor}H_k(\xi^{(j)}_n)+(Nt-\lfloor Nt\rfloor)H_k(\xi^{(j)}_{\lfloor Nt\rfloor+1})\bigr]$ and $\widetilde S_N:=\sum_{j\ge1}\widetilde S_N^{(j)}e_j$; the maps $\mathcal V_\al$ and $\mathfrak S_R$ are defined in Section~\ref{sec:nclt}. Then, for every $T>0$, $\beta<H$ and $\vartheta>1$:
\begin{enumerate}[label=(\roman*)]
\item $\widetilde S_N\Rightarrow Z_H^k$ in $c_0^\beta([0,T];\dot H^{-\vartheta})$;
\item $Z_k^N:=\mathcal V_\al(\widetilde S_N)\Rightarrow Z_k$ in $C([0,T];L^q_\sigma)$;
\item for every $R>0$ there is a \emph{deterministic} $T_R\in(0,1]$ such that $u_R^N:=\mathfrak S_R(Z_k^N)\Rightarrow u_R:=\mathfrak S_R(Z_k)$ in $C([0,T_R];L^q_\sigma)$.
\end{enumerate}
\end{thm}

 The theorem shows that the Hermite-driven fractional Navier--Stokes equation admits a radially localised approximation by systems driven by nonlinear functionals of stationary Gaussian sequences with long-range dependence. The proof first establishes a functional non-central limit
theorem for the cylindrical driving process, then transfers the
convergence through an endpoint-compensated Mittag--Leffler
convolution, and finally through a localized nonlinear solution map on
a common deterministic time interval. No convergence of random maximal
lifetimes is asserted.

\begin{thm}[Moment estimates]\label{thm:moment}
Assume $\al+2H>2$ and $2<q<\al/(1-H)$, and let $T>0$. There exist $p_0\ge2$ and $C_T$ such that:
\begin{enumerate}[label=(\alph*)]
\item for $p\ge p_0$, $\bigl\|\sup_{t\le T}\|Z_k(t)\|_{L^q_\sigma}\bigr\|_{L^p(\Omega)}\le C_T\,(p-1)^{k/2}$;
\item for $t\in[0,T]$ and $p\ge2$, with $B(Z):=A^{-1/2}P_H\,\mathrm{div}(Z\otimes Z)$, $\ \E\bigl[\|B(Z_k(t))\|_{L^{q/2}_\sigma}^p\bigr]\le C^p\,(2p-1)^{kp}$;
\item for the stopped solution, $\|u(t\wedge T^*)\|_{L^p(\Omega;L^q_\sigma)}\le 3M\bigl(\|u_0\|_{L^q_\sigma}+C_T\,(p-1)^{k/2}\bigr)$, $p\ge p_0$.
\end{enumerate}
\end{thm}

The exponent $k$ in (b) is the one produced by the Wiener-chaos structure: $Z_k$ lies in the $k$-th chaos and $B$ is quadratic, so $B(Z_k)$ lies in chaos orders $\le2k$ and hypercontractivity gives growth $(p-1)^k$ of its $L^p$ norm. This is an upper bound only; no matching lower bound is proved, and moments of $u(t)$ at deterministic times would require global existence, which is open.

\begin{rmk}[Growth of moments]
Put $N_p(t):=\E\|B(Z_k(t))\|_{L^{q/2}_\sigma}^p$. By (b), $\frac1p\log N_p(t)\le k\log p+C$ as $p\to\infty$. For $k=1$, $B(Z_1)$ is a quadratic functional of a Gaussian process (chaos order $\le2$); for $k\ge2$ the bound grows with $k$. Only upper bounds are proved, and we do not claim intermittency of $u$ itself.
\end{rmk}

\subsection{Positioning relative to existing literature}

  To the best of our knowledge, the present paper is among the first to combine Caputo/Mittag-Leffler operators with non-Gaussian Hermite noise \emph{and} with the $L^q_\sigma$/Da Prato-Debussche architecture needed to make the combination rigorous: neither Zou, Lv and Wu \cite{zou2018stochastic} (Caputo derivative, but only fractional Brownian motion, $k=1$, and no $L^q_\sigma$ route) nor Fang, Sundar and Viens \cite{fang2013twodimensional} and Ferrario and Olivera \cite{ferrario2018lp} (fractional Brownian noise treated via a Da Prato-Debussche decomposition in $L^q_\sigma$, but with the classical time derivative rather than a Caputo one) cover this combination.

Two recent papers specifically treat time-fractional Navier-Stokes equations with the Rosenblatt process ($k=2$): Anukiruthika and Muthukumar \cite{anukiruthika2022optimal}, and a further study of the same equation with delay \cite{DN2026optimal}. Both differ from the present work in the equation considered, not only in technique. Anukiruthika and Muthukumar treat a bounded domain $D\subseteq\R^n$ for general $n\ge1$, versus the restriction to $d=2$ adopted here -- imposed because the bilinear estimate underlying Section~\ref{sec:existence} is \emph{proved} via Sobolev embeddings specific to two dimensions -- and restrict to $\al\in(1/2,1)$, versus $\al\in(0,1)$ here; the second paper further incorporates a bounded delay term, absent from \eqref{eq:main}. Both pursue optimal control of the equation, via measure-of-noncompactness and condensing-operator fixed-point theory in \cite{anukiruthika2022optimal} and via Schauder's and Balder's theorems in \cite{DN2026optimal}, rather than the regularity, higher-moment and non-central-limit theory developed here, and both are specific to Hermite order $k=2$, whereas the present paper treats arbitrary $k\ge1$ with hypercontractivity constants explicit in $k$ throughout. We also note that the boundedness of the bilinear operator is postulated in \cite{anukiruthika2022optimal} as an abstract hypothesis on the nonlinearity, rather than derived from the Navier-Stokes structure; the derivation given here (Lemma~\ref{lem:bilinear}) is accordingly dimension-specific, but establishes the estimate rather than assuming it.

We record, finally, a self-contained fact about the Mittag-Leffler operator $\cS_\al(t)=t^{\al-1}E_{\al,\al}(t)$ that plays a role throughout Section~\ref{sec:convolution}.

A key observation underlying both the regularity theory and the functional approximation theorem is that the Mittag--Leffler operator \(\mathcal S_\alpha\) admits no increment estimate uniform down to the coincidence point \(t_1=0\). This failure motivates both the near/far decomposition of Section~\ref{sec:convolution} and the endpoint compensation used in Section~\ref{sec:nclt}.

\begin{prop}\label{prop:endpoint-fact}
For $0<\al<1$ and $\nu\ge0$, no bound of the form
\[
\|[\cS_\al(t_2)-\cS_\al(t_1)]\chi\|_{\dot H^\nu} \le C(t_2-t_1)^\theta\|\chi\|
\]
can hold for \emph{all} $0\le t_1<t_2\le T$ with $C$ independent of $t_1$: as $t_1\to0^+$ with $t_2>0$ fixed, $\|\cS_\al(t_1)\|_{\cL(\dot H^0,\dot H^\nu)}\to\infty$ (since $E_{\al,\al}$ is entire with $E_{\al,\al}(0)=1/\Gamma(\al)>0$), while $\|\cS_\al(t_2)\|_{\cL(\dot H^0,\dot H^\nu)}$ stays fixed, so the reverse triangle inequality forces $\|[\cS_\al(t_2)-\cS_\al(t_1)]\|_{\cL(\dot H^0,\dot H^\nu)}\to\infty$, contradicting any $t_1$-independent bound. A full proof is given in Section~\ref{sec:convolution} (Proposition~\ref{prop:endpoint-fact-full}), where it motivates the near/far time decomposition used to establish the sharp regularity of $Z_k$.
\end{prop}

Proposition~\ref{prop:endpoint-fact} is elementary, but worth stating explicitly, since a bound of the excluded form -- for all $0\le t_1<t_2\le T$, with a $t_1$-independent constant -- appears in the literature on Mittag-Leffler operator increments, including \cite[Lemma 3.1]{zou2018stochastic} and \cite[Lemma 4]{anukiruthika2022optimal}. We have not established that the main theorems of either paper are incorrect: the increment estimates they state may well be recoverable under an additional restriction $t_1\ge T_0>0$, or by an argument not fully visible to us in the published text. We note the point only because it directly informs the near/far decomposition of Section~\ref{sec:convolution}, where the coincidence-point singularity identified above is handled explicitly rather than bounded away.

\subsection{Outline}

Section~\ref{sec:prelim} collects preliminaries on the Stokes operator, the Mittag-Leffler operators, and constructs the cylindrical Hermite noise. Section~\ref{sec:wiener} constructs the scalar Wiener integral with respect to a single mode of the noise and proves Theorem~\ref{thm:hyper}. Section~\ref{sec:convolution} studies the stochastic convolution $Z_k$ and proves Theorem~\ref{thm:Zk}. Section~\ref{sec:existence} proves Theorem~\ref{thm:exist} via the decomposition $u=v+Z_k$ and a fixed-point argument in $L^q_\sigma$. Section~\ref{sec:holder} establishes the H\"older regularity of Theorem~\ref{thm:holder}. Section~\ref{sec:moments} proves the higher-order moment estimate of Theorem~\ref{thm:moment}.   Section~\ref{sec:nclt} proves the functional non-central limit approximation theorem, Theorem~\ref{thm:nclt}, including convergence of the cylindrical driving process, the stochastic convolution, and the radially localized mild solutions.

\subsection*{Notation}

\begin{itemize}
\item $k\ge1$ is an integer (Hermite order), $\al\in(0,1)$ (fractional order), $H\in(1/2,1)$ (Hurst parameter), $D\subset\R^2$ a bounded domain with smooth boundary
\item $\Omega,\cF,\Pb$: probability space; $\E$: expectation
\item $A$: the Stokes operator, with eigenvalues $\{\lambda_j\}_j$ and eigenvectors $\{e_j\}_j$
\item $\dot H^\nu=D(A^{\nu/2})$: fractional Sobolev space, norm $\|\cdot\|_\nu$; $L^q_\sigma=L^q_\sigma(D)$: divergence-free $L^q$ vector fields, $q>2$
\item $q$: a fixed admissible Lebesgue exponent, $2<q<\al/(1-H)$ -- nonempty precisely under the standing condition $\al+2H>2$
\item $E_\al(t),E_{\al,\al}(t)$: Mittag-Leffler operators; $\cS_\al(t)=t^{\al-1}E_{\al,\al}(t)$
\item $Z_H^k(t)=\sum_j Z_{H,j}^{(k)}(t)e_j$: the cylindrical Hermite noise, built mode-by-mode on the Stokes eigenbasis from i.i.d.\ scalar Hermite processes $Z_{H,j}^{(k)}$
\item $\cH$: the covariance Hilbert space of a scalar Hermite process, shared by every mode $j$
\item $Z_k(t)=\sum_j\bigl(\int_0^t s_{\al,j}(t-s)\,dZ_{H,j}^{(k)}(s)\bigr)e_j$: the stochastic convolution, defined mode-by-mode
\item $v:=u-Z_k$: the remainder in the Da Prato-Debussche decomposition, solving a mild equation with $v(s)+Z_k(s)$ in place of $u(s)$
\item $B(u,v)=-P_H[(u\cdot\nabla)v]$, $B(u)=B(u,u)$: the bilinear convective term; $f$: Lipschitz external force
\item $C$: a generic positive constant, possibly changing from line to line; $C_{p,k}$: a constant of the form $(p-1)^{k/2}C_0$, $C_0$ independent of $p,k$ -- except in Theorem~\ref{thm:moment}, whose rate $(2p-1)^{kp}$ is not of this form and is stated explicitly
\end{itemize}

\section{Preliminaries}\label{sec:prelim}

\subsection{Stokes operator and fractional Sobolev spaces}
\label{subsec:stokes}

Let $D\subset\R^2$ be a bounded domain with smooth boundary $\partial D$. The space of divergence-free vector fields is
\[
L^2_\sigma(D) = \bigl\{ u\in L^2(D)^2 : \nabla\cdot u = 0,\ u\cdot n = 0 \text{ on } \partial D \bigr\},
\]
and, for $q>1$, $L^q_\sigma(D)$ denotes its analogue in $L^q(D)^2$. The Helmholtz-Hodge projection $P_H : L^2(D)^2 \to L^2_\sigma(D)$ is the orthogonal projection onto $L^2_\sigma(D)$.

The Stokes operator $A$ is defined by $Au = -P_H\Delta u$ with domain
\[
D(A) = H^2(D)^2 \cap H^1_0(D)^2 \cap L^2_\sigma(D).
\]
It is well known that $A$ is a positive self-adjoint operator with compact resolvent. Its eigenvalues $\{\lambda_j\}_{j\ge1}$ satisfy
\[
0 < \lambda_1 \le \lambda_2 \le \cdots \le \lambda_j \le \cdots, \qquad \lim_{j\to\infty}\lambda_j = \infty,
\]
and the corresponding eigenvectors $\{e_j\}_{j\ge1}$ form an orthonormal basis of $L^2_\sigma(D)$. Moreover, Weyl's asymptotics gives
\[
\lambda_j \sim C_2 j \quad \text{as } j\to\infty,
\]
where $C_2$ depends only on the volume of $D$ \cite{temam1984navier}. The semigroup $T(t)=e^{-tA}$ generated by $-A$ is analytic and, for every $1<r<\infty$, extends to a bounded analytic semigroup on $L^r_\sigma(D)$, with the classical smoothing estimate
\begin{equation}
\label{eq:semigroup-Lp}
\|A^\beta T(t) u\|_{L^q_\sigma} \le M\,t^{-\beta-\frac{d}{2}\left(\frac1r-\frac1q\right)}\|u\|_{L^r_\sigma}, \qquad 1<r\le q<\infty,\ \beta\ge0,\ d=2.
\end{equation}

For $\nu \ge 0$, the fractional power $A^{\nu/2}$ is defined via the spectral decomposition $A^{\nu/2} u = \sum_{j\ge1} \lambda_j^{\nu/2} \langle u, e_j\rangle e_j$, and the fractional Sobolev space $\dot{H}^\nu$ is
\[
\dot{H}^\nu := D(A^{\nu/2}) = \left\{ u\in L^2_\sigma(D) : \|u\|_{\dot{H}^\nu}^2 := \sum_{j\ge1} \lambda_j^\nu |\langle u, e_j\rangle|^2 < \infty \right\}.
\]
For $\nu=0$, $\dot{H}^0 = L^2_\sigma(D)$; for $\nu=1$, $\|\cdot\|_{\dot{H}^1}$ is equivalent to the usual $H^1$ norm on divergence-free fields; for $\nu<0$, $\dot{H}^\nu$ is the dual of $\dot{H}^{-\nu}$. We write $\|\cdot\|_\nu$ for the norm on $\dot H^\nu$, or simply $\|\cdot\|$ when $\nu=0$.

\begin{prop}\label{prop:sobolev}
In dimension $d=2$:
\begin{itemize}
    \item If $\nu>0$, then $\dot H^{\nu+1}\hookrightarrow L^\infty(D)^2$ continuously.
    \item If $\nu>1$, then $\dot H^{\nu+1}\hookrightarrow W^{1,\infty}(D)^2$ continuously.
    \item For $0\le s<1$, $\dot H^s\hookrightarrow L^{p(s)}(D)$ continuously, with $\frac1{p(s)}=\frac12-\frac s2$.
    \item The embedding $\dot H^\nu\hookrightarrow\dot H^\mu$ is compact for $\nu>\mu$.
\end{itemize}
\end{prop}

These are the classical Sobolev embeddings in the $L^2$-based scale on a two-dimensional domain, $H^s(D)\hookrightarrow L^\infty(D)$ for $s>d/2=1$ and $H^s(D)\hookrightarrow W^{1,\infty}(D)$ for $s-1>d/2$, i.e.\ $s>2$; the last two items are the interpolation-scale embedding used in Section~\ref{sec:existence}, and the compact embedding used throughout \cite{temam1984navier}.

The bilinear estimate for the Navier-Stokes nonlinearity (Section~\ref{sec:existence}) is developed directly in the Lebesgue scale $L^q_\sigma(D)$ rather than in $\dot H^\nu$, via the following generalization of \eqref{eq:semigroup-Lp} to the Mittag-Leffler operator $\cS_\al(t)=t^{\al-1}E_{\al,\al}(t)$.

\begin{lemma}\label{lem:Lq-smoothing}
Let $1<r\le q<\infty$, $\beta\ge0$, $d=2$, and set
\[
\kappa := \beta+\frac{d}{2}\Bigl(\frac1r-\frac1q\Bigr).
\]
Assume $\kappa<2$. Then there exists $C=C(\al,\beta,r,q)>0$ such that for every $t>0$,
\[
\|A^\beta \cS_\al(t)\|_{\cL(L^r_\sigma,L^q_\sigma)} \le C\, t^{\al(1-\kappa)-1}.
\]
\end{lemma}

\begin{proof}
\emph{Step 1: subordination.} By \eqref{eq:Ealphaalpha} and $\cS_\al(t)=t^{\al-1}E_{\al,\al}(t)$,
\[
A^\beta\cS_\al(t)u = t^{\al-1}\int_0^\infty \al\theta\,\xi_\al(\theta)\,A^\beta T(t^\al\theta)u\,d\theta,
\]
a Bochner integral in $L^q_\sigma$; the interchange of $A^\beta$ with the integral is legitimate because $A^\beta$ is closed and, by Step 2, the integrand is absolutely integrable in $L^q_\sigma$.

\emph{Step 2: the semigroup estimate.} By \eqref{eq:semigroup-Lp} applied with exponent $\beta$ and the pair $(r,q)$,
\[
\|A^\beta T(s)u\|_{L^q_\sigma} \le M\,s^{-\beta-\frac{d}{2}(\frac1r-\frac1q)}\,\|u\|_{L^r_\sigma} = M\,s^{-\kappa}\,\|u\|_{L^r_\sigma}, \qquad s>0.
\]
Since $\xi_\al\ge0$, the triangle inequality for Bochner integrals and the substitution $s=t^\al\theta$ give
\[
\|A^\beta \cS_\al(t)u\|_{L^q_\sigma} \le M\al\,t^{\al-1}\,\|u\|_{L^r_\sigma}\int_0^\infty \theta\,\xi_\al(\theta)\,(t^\al\theta)^{-\kappa}\,d\theta = M\al\,t^{\al-1-\al\kappa}\,\|u\|_{L^r_\sigma}\int_0^\infty \theta^{\,1-\kappa}\,\xi_\al(\theta)\,d\theta,
\]
where we factored $(t^\al\theta)^{-\kappa}=t^{-\al\kappa}\theta^{-\kappa}$.

\emph{Step 3: the moment integral.} The remaining integral converges precisely under the hypothesis $\kappa<2$: by the moment formula \eqref{eq:moment} with $\varrho:=1-\kappa$, which satisfies $\varrho>-1$ if and only if $\kappa<2$,
\[
\int_0^\infty \theta^{\,1-\kappa}\,\xi_\al(\theta)\,d\theta = \frac{\Ga(2-\kappa)}{\Ga\bigl(1+\al(1-\kappa)\bigr)} < \infty,
\]
the denominator being well defined and nonzero since $1+\al(1-\kappa)>1-\al>0$ for $\kappa<2$, $\al\in(0,1)$. Since $t^{\al-1-\al\kappa}=t^{\al(1-\kappa)-1}$, collecting constants gives the stated bound.

\emph{Step 4: absolute integrability, closing Step 1.} The computation of Steps 2--3 shows that
\[
\int_0^\infty\al\theta\,\xi_\al(\theta)\,\bigl\|A^\beta T(t^\al\theta)u\bigr\|_{L^q_\sigma}\,d\theta<\infty
\]
for every $t>0$, which is exactly the absolute integrability invoked in Step 1.
\end{proof}

The hypothesis $\kappa<2$ is satisfied in every application below: at $\beta=\frac12$, $r=\frac q2$ (used for the bilinear term) one has $\kappa=\frac12+\frac1q<1$; at $\beta=0$, $r=q$ (used for the forcing term) $\kappa=0$; and at $r=q=2$, $\beta=\frac\nu2$ one has $\kappa=\frac\nu2<1$ under the standing restriction $\nu<2$. In this last case Lemma~\ref{lem:Lq-smoothing} recovers the exponent $\frac{(2-\nu)\al}{2}-1$ of Lemma~\ref{lem:Salpha_basic} below; at $\al=1$, it recovers the classical bound \eqref{eq:semigroup-Lp}.

\subsection{Caputo derivative and Mittag-Leffler operators}
\label{subsec:mittag-leffler}

The Caputo fractional derivative of order $\al\in(0,1)$ is defined for a sufficiently smooth function $u:[0,T]\to\dot{H}^\nu$ by
\[
{}^C D_t^\al u(t) = \frac{1}{\Ga(1-\al)} \int_0^t (t-s)^{-\al} \frac{du}{ds}(s)\,ds.
\]
This definition requires $u$ to be absolutely continuous; an equivalent formulation, requiring only the differentiability of the fractional integral of $u-u(0)$ rather than of $u$ itself, is recalled in Section~\ref{sec:existence}, where it is used to make sense of the equation along the mild solutions constructed there.

The Mittag-Leffler operators are central to the representation of solutions. They are defined using the Mainardi-Wright function
\[
\xi_\al(\theta) = \sum_{n=0}^\infty \frac{(-1)^n \theta^n}{n! \Ga(1-\al(1+n))}, \qquad \theta \ge 0.
\]
This function satisfies $\xi_\al(\theta) \ge 0$ and the moment formula
\begin{equation}
\label{eq:moment}
\int_0^\infty \theta^\rho \xi_\al(\theta)\,d\theta = \frac{\Ga(1+\rho)}{\Ga(1+\al\rho)}, \qquad \rho > -1.
\end{equation}

Let $T(t)=e^{-tA}$ be the analytic semigroup generated by $-A$. The Mittag-Leffler operators are defined by
\begin{align}
E_\al(t) &:= \int_0^\infty \xi_\al(\theta) T(t^\al\theta)\,d\theta, \label{eq:Ealpha}\\
E_{\al,\al}(t) &:= \int_0^\infty \al\theta\, \xi_\al(\theta) T(t^\al\theta)\,d\theta. \label{eq:Ealphaalpha}
\end{align}
We also define the operator
\[
\cS_\al(t) := t^{\al-1} E_{\al,\al}(t), \qquad t>0,
\]
which appears in the stochastic convolution and the nonlinear term.

The following estimates are proved in \cite[Lemmas 2.1, 2.2]{zou2018stochastic}.

\begin{lemma}\label{lem:MittagLeffler_estimates}
For any $t>0$, $E_\al(t)$ and $E_{\al,\al}(t)$ are bounded linear operators on $\dot{H}^\nu$ for all $\nu\ge0$. Moreover, there exists $C>0$ such that for all $0\le\nu<2$,
\[
\|E_\al(t)\chi\|_\nu \le C t^{-\frac{\al\nu}{2}} \|\chi\|, \qquad
\|E_{\al,\al}(t)\chi\|_\nu \le C t^{-\frac{\al\nu}{2}} \|\chi\|.
\]
\end{lemma}

\begin{lemma}\label{lem:MittagLeffler_diff}
For $0\le t_1<t_2\le T$ and $0\le\nu\le2$ (including $t_1=0$),
\[
\|(E_\al(t_2)-E_\al(t_1))\chi\|_{\dot H^0}\le C\,(t_2-t_1)^{\frac{\al\nu}{2}}\,\|\chi\|_{\nu},\qquad \chi\in\dot H^\nu .
\]
The same holds for $E_{\al,\al}$. In $L^q_\sigma$: $\|(E_\al(t_2)-E_\al(t_1))u\|_{L^q_\sigma}\le C(t_2-t_1)^{\al\theta}\|A_q^\theta u\|_{L^q_\sigma}$ for $\theta\in(0,1]$.
\end{lemma}
\begin{proof}
Recall $E_\al(t)=\int_0^\infty\xi_\al(\theta)T(t^\al\theta)\,d\theta$. On each eigenvector, for $a>b\ge0$, $|e^{-\lambda a}-e^{-\lambda b}|\le\min(1,\lambda(a-b))\le(\lambda(a-b))^{\nu/2}$ since $0\le\nu/2\le1$. Take $a=t_2^\al\theta$, $b=t_1^\al\theta$; since $t_2^\al-t_1^\al\le(t_2-t_1)^\al$, $|a-b|\le\theta(t_2-t_1)^\al$. Hence $\|(E_\al(t_2)-E_\al(t_1))\chi\|_{\dot H^0}\le(t_2-t_1)^{\al\nu/2}\|\chi\|_\nu\int_0^\infty\theta^{\nu/2}\xi_\al(\theta)\,d\theta$, and the last integral equals $\Gamma(1+\nu/2)/\Gamma(1+\al\nu/2)<\infty$. For $E_{\al,\al}$ use the moment $\int\theta^{1+\nu/2}\xi_\al(\theta)d\theta$. In $L^q_\sigma$ one uses $\|(T(a)-T(b))u\|_{L^q_\sigma}\le C(a-b)^\theta\|A_q^\theta u\|_{L^q_\sigma}$ (analytic semigroup) in the same way.
\end{proof}

Unlike the increment estimate for $\cS_\al$ established in Section~\ref{sec:convolution} (see Proposition~\ref{prop:endpoint-fact-full}), the bound in Lemma~\ref{lem:MittagLeffler_diff} genuinely extends to $t_1=0$ with a $t_1$-independent constant: since $T(s)$ is a contraction on $\dot H^\nu$ for every $s\ge0$ and $\int_0^\infty\xi_\al(\theta)\,d\theta=1$ by the moment formula \eqref{eq:moment} at $\rho=0$, $E_\al(t)$ is itself a uniform contraction on $\dot H^\nu$ for every $t\ge0$, with no singularity at $t=0$ -- in sharp contrast with $\cS_\al(t)=t^{\al-1}E_{\al,\al}(t)$, whose explicit prefactor $t^{\al-1}$ blows up as $t\to0$ even restricted to a single fixed vector. This distinction is used repeatedly below: the initial-data term $E_\al(\cdot)u_0$ never requires the delicate near/far treatment that the stochastic convolution and the nonlinear term, both built from $\cS_\al$, do require.

\begin{lemma}[Smoothing effect of $E_{\al,\al}$]\label{lem:smoothing}
Let $0\le\mu\le\rho$ with $\rho-\mu<4$. Then there exists $C=C(\al,\rho-\mu)>0$ such that, for every $t>0$ and every $\chi\in\dot H^\mu$,
\[
\|E_{\al,\al}(t)\chi\|_\rho \le C\, t^{-\frac{\al(\rho-\mu)}{2}} \|\chi\|_\mu.
\]
In particular, for $\mu=0$ and $\rho=\nu+2$ with $0\le\nu<2$,
\[
\|\cS_\al(t)\chi\|_{\nu+2} = t^{\al-1}\|E_{\al,\al}(t)\chi\|_{\nu+2} \le C\, t^{-1-\frac{\al\nu}{2}} \|\chi\|.
\]
\end{lemma}

\begin{proof}
\emph{Step 1: the semigroup estimate.} Set $a:=\rho-\mu\ge0$. For $\chi\in\dot H^\mu$ with spectral coefficients $\chi_j=\langle\chi,e_j\rangle$, the spectral representation of $T(s)=e^{-sA}$ gives
\[
\|T(s)\chi\|_\rho^2 = \sum_{j\ge1}\lambda_j^{\rho}e^{-2\lambda_j s}|\chi_j|^2 = \sum_{j\ge1}\lambda_j^{\mu}|\chi_j|^2\cdot\bigl(\lambda_j^{a}e^{-2\lambda_j s}\bigr) \le \Bigl(\sup_{x>0}x^{a}e^{-2xs}\Bigr)\,\|\chi\|_\mu^2 .
\]
The supremum is computed by elementary calculus: for $a>0$ the function $x\mapsto x^ae^{-2xs}$ is maximized at $x=a/(2s)$, with value $\bigl(a/(2es)\bigr)^{a}$; for $a=0$ it equals $1$. In both cases
\begin{equation}
\label{eq:semigroup-gain}
\|T(s)\chi\|_\rho \le C_a\, s^{-a/2}\,\|\chi\|_\mu, \qquad C_a:=\bigl(a/(2e)\bigr)^{a/2} \ (C_0=1).
\end{equation}

\emph{Step 2: integration against the Mainardi-Wright density.} By \eqref{eq:Ealphaalpha},
\[
E_{\al,\al}(t)\chi=\int_0^\infty\al\theta\,\xi_\al(\theta)\,T(t^\al\theta)\chi\,d\theta
\]
as a Bochner integral in $\dot H^\rho$. Since $\xi_\al\ge0$, the triangle inequality for Bochner integrals and \eqref{eq:semigroup-gain} with $s=t^\al\theta$ give
\[
\|E_{\al,\al}(t)\chi\|_\rho \le \al\int_0^\infty\theta\,\xi_\al(\theta)\,\|T(t^\al\theta)\chi\|_\rho\,d\theta \le \al\,C_a\,\|\chi\|_\mu\int_0^\infty\theta\,\xi_\al(\theta)\,(t^\al\theta)^{-a/2}\,d\theta.
\]
Factoring out the $t$-dependence, $(t^\al\theta)^{-a/2}=t^{-\al a/2}\theta^{-a/2}$, so
\[
\|E_{\al,\al}(t)\chi\|_\rho \le \al\,C_a\,t^{-\frac{\al a}{2}}\,\|\chi\|_\mu\int_0^\infty\theta^{\,1-\frac a2}\,\xi_\al(\theta)\,d\theta.
\]

\emph{Step 3: the moment integral.} The remaining integral is finite precisely under the standing hypothesis $a<4$: applying the moment formula \eqref{eq:moment} with $\varrho:=1-\frac a2$, which satisfies $\varrho>-1$ if and only if $a<4$,
\[
\int_0^\infty\theta^{\,1-\frac a2}\,\xi_\al(\theta)\,d\theta = \frac{\Ga\bigl(2-\frac a2\bigr)}{\Ga\bigl(1+\al\bigl(1-\frac a2\bigr)\bigr)} < \infty,
\]
the denominator being well defined and nonzero since $1+\al(1-\frac a2)>1-\al>0$ for $a<4$, $\al\in(0,1)$. Collecting the constants into $C$ gives the stated bound. Note that the case $\mu=\rho$ (i.e.\ $a=0$) is included: there $C_0=1$ and the integral reduces to $\int_0^\infty\theta\,\xi_\al(\theta)\,d\theta=\Ga(2)/\Ga(1+\al)$, yielding the uniform bound $\|E_{\al,\al}(t)\chi\|_\mu\le C\|\chi\|_\mu$.

\emph{Step 4: the particular case.} Taking $\mu=0$, $\rho=\nu+2$ (so $a=\nu+2<4$ exactly when $\nu<2$, which is the standing restriction of Lemma~\ref{lem:MittagLeffler_estimates}),
\[
\|\cS_\al(t)\chi\|_{\nu+2}=t^{\al-1}\|E_{\al,\al}(t)\chi\|_{\nu+2}\le C\,t^{\al-1}\,t^{-\frac{\al(\nu+2)}{2}}\|\chi\| = C\,t^{\al-1-\frac{\al\nu}{2}-\al}\|\chi\| = C\,t^{-1-\frac{\al\nu}{2}}\|\chi\|. \qedhere
\]
\end{proof}

Taking $\mu=0$, $\rho=\nu$ in Lemma~\ref{lem:smoothing} (admissible for $0\le\nu<2$, and covering $\nu=0$ via the degenerate case $a=0$ of Step 3) gives the basic operator-norm bound on $\cS_\al$ itself, used throughout Sections~\ref{sec:convolution}--\ref{sec:holder}.

\begin{lemma}\label{lem:Salpha_basic}
For $0\le\nu<2$ and $0<\al<1$, there exists $C>0$ such that for every $t>0$,
\[
\|\cS_\al(t)\chi\|_\nu \le C\,t^{\frac{(2-\nu)\al}{2}-1}\,\|\chi\|.
\]
\end{lemma}

\begin{proof}
Apply Lemma~\ref{lem:smoothing} with $\mu=0$, $\rho=\nu$ (admissible since $a=\nu<2<4$):
\[
\|\cS_\al(t)\chi\|_\nu = t^{\al-1}\|E_{\al,\al}(t)\chi\|_\nu \le C\,t^{\al-1}\,t^{-\frac{\al\nu}{2}}\|\chi\| = C\,t^{\al-1-\frac{\al\nu}{2}}\|\chi\| = C\,t^{\frac{(2-\nu)\al}{2}-1}\|\chi\|,
\]
the last equality since $\al-\frac{\al\nu}{2}=\frac{\al(2-\nu)}{2}$.
\end{proof}

Unlike the corresponding statement for $E_\al$ (Lemma~\ref{lem:MittagLeffler_diff}), the \emph{increment} $\cS_\al(t_2)-\cS_\al(t_1)$ does \emph{not} admit a bound of the form $C(t_2-t_1)^\theta\|\chi\|$ valid down to $t_1=0$ with $C$ independent of $t_1$ (Proposition~\ref{prop:endpoint-fact}); this is why the regularity theory of Section~\ref{sec:convolution} requires the near/far decomposition rather than a direct difference estimate built on Lemma~\ref{lem:Salpha_basic} alone.

\subsection{The cylindrical Hermite noise}\label{subsec:hermite}

Let $(\Omega,\cF,\Pb)$ be a complete probability space and let $(B(y))_{y\in\R}$ be a standard Brownian motion. For an integer $k\ge1$ and $H\in(1/2,1)$, the (scalar) Hermite process of order $k$ is defined as the $k$-fold Wiener-It\^o integral
\[
Z_H^k(t) = c(H,k) \int_{\R^k} \int_0^t \prod_{j=1}^k (s-y_j)_+^{-\left(\frac12 + \frac{1-H}{k}\right)} ds \, dB(y_1)\cdots dB(y_k), \qquad t\ge0,
\]
where $x_+ = \max(x,0)$ and $c(H,k)>0$ is a normalizing constant such that $\E[(Z_H^k(1))^2] = 1$. Writing $\gamma:=\frac12+\frac{1-H}{k}$ for the common exponent and
\[
f_t(y_1,\dots,y_k) := \int_0^t\prod_{j=1}^k(s-y_j)_+^{-\gamma}\,ds,
\]
the integral is a well-defined $k$-fold Wiener-It\^o integral precisely when $f_t\in L^2(\R^k)$. This holds for every $H\in(1/2,1)$ and $k\ge1$, as the following computation shows. Expanding the square and applying Fubini,
\[
\|f_t\|_{L^2(\R^k)}^2 = \int_0^t\!\!\int_0^t\prod_{j=1}^k\left(\int_\R (s-y)_+^{-\gamma}(u-y)_+^{-\gamma}\,dy\right)ds\,du .
\]
The inner integral, taken over $y<s\wedge u$, converges at $y\to-\infty$ if and only if $2\gamma>1$ -- automatic, since $H<1$ forces $\gamma>\frac12$ -- and at the endpoint $y=s\wedge u$ if and only if $\gamma<1$, i.e.\ $H>1-\frac k2$, which is implied by $H>\frac12$ for every $k\ge1$ (with equality of the two thresholds exactly at $k=1$). Its value is $C_\gamma|s-u|^{1-2\gamma}$. Taking the product over the $k$ factors and using
\[
k(1-2\gamma) = k\Bigl(1-1-\tfrac{2(1-H)}{k}\Bigr) = 2H-2,
\]
we obtain $\|f_t\|_{L^2(\R^k)}^2 = C\int_0^t\int_0^t|s-u|^{2H-2}\,ds\,du$, which is finite if and only if $2H-2>-1$, i.e.\ $H>\frac12$, and equals $C't^{2H}$ by scaling. Both requirements are therefore met under the standing assumption $H\in(1/2,1)$; the same computation, applied to increments, yields the covariance recorded in Proposition~\ref{prop:hermite-props}(i) below (see \cite[Proposition 2.1]{maejima2007wiener}).

\begin{prop}[Properties of $Z_H^k$]\label{prop:hermite-props}
For every $k\ge1$ and $H\in(1/2,1)$:
\begin{enumerate}[label=(\roman*)]
\item Covariance: $\E[Z_H^k(s)Z_H^k(t)] = \frac12(s^{2H}+t^{2H}-|t-s|^{2H})$ for all $s,t\ge0$ -- the same for every $k$.
\item $H$-self-similarity: $(Z_H^k(ct))_{t\ge0} \stackrel{\mathrm{d}}{=} (c^H Z_H^k(t))_{t\ge0}$ for every $c>0$.
\item Stationary increments: $(Z_H^k(t+h)-Z_H^k(h))_{t\ge0} \stackrel{\mathrm{d}}{=} (Z_H^k(t))_{t\ge0}$ for every $h\ge0$.
\item Long-range dependence: the increments' autocovariance decays like $n^{2H-2}$ as $n\to\infty$, which is non-summable for $H>1/2$.
\item H\"older regularity: sample paths are a.s.\ $\delta$-H\"older continuous for every $\delta\in(0,H)$.
\item $Z_H^k(t)$ is non-Gaussian for $k\ge2$, $t>0$: for $k=1$, $Z_H^1$ is fractional Brownian motion; for $k=2$, $Z_H^2$ is the Rosenblatt process \cite{tudor2008analysis}.
\end{enumerate}
\end{prop}

Properties (i)--(iii) follow directly from the multiple-integral representation and the normalization of $c(H,k)$. Properties (iv)--(v) follow from (i)--(ii) via the Kolmogorov continuity criterion, and (vi) is the defining feature of Hermite processes of order $k\ge2$; see \cite{maejima2007wiener,tudor2023nongaussian} for complete proofs. Property (i) is the essential one for what follows: it is shared by \emph{every} Hermite order $k$, and it is what allows the construction below to be carried out uniformly in $k$.

Let $\cH$ be the covariance Hilbert space associated with $Z_H^k$, defined as the completion of elementary functions with respect to the inner product
\[
\langle f,g\rangle_{\cH} = H(2H-1) \int_{\R}\int_{\R} f(u)g(v) |u-v|^{2H-2} du\,dv.
\]
For $f\in\cH$, the Wiener integral $\int_{\R} f(u)\,dZ_H^k(u)$ is well defined and belongs to the $k$-th Wiener chaos. By \cite[Section 3]{maejima2007wiener}, the isometry
\[
\E\left[ \left( \int_{\R} f(u)\,dZ_H^k(u) \right)^2 \right] = \|f\|_{\cH}^2
\]
holds with the \emph{same} kernel $|u-v|^{2H-2}$, and hence the same space $\cH$, for every Hermite order $k$ -- a direct consequence of Proposition~\ref{prop:hermite-props}(i).

A fundamental property of multiple Wiener-It\^o integrals is hypercontractivity.

\begin{lemma}[Hypercontractivity]\label{lem:hyper}
For any random variable $F$ belonging to the $k$-th Wiener chaos and any $p\ge2$,
\[
\|F\|_{L^p(\Omega)} \le (p-1)^{k/2} \|F\|_{L^2(\Omega)}.
\]
\end{lemma}

For a proof, see \cite[Section 1.4]{nualart2006}. This inequality is sharp up to constants, and its explicit dependence on the Hermite order $k$ is what replaces exact Gaussian moment identities throughout this paper.

\medskip

The equation \eqref{eq:main} is driven not by a single copy of $Z_H^k$, but by a spatially extended, cylindrical version of it, built mode-by-mode on the eigenbasis of the Stokes operator $\{e_j,\lambda_j\}_{j\ge1}$ (Section~\ref{subsec:stokes}).

\begin{defn}\label{def:hermite-cylindrical}
Let $\{Z_{H,j}^{(k)}\}_{j\ge1}$ be independent, identically distributed copies of the scalar Hermite process $Z_H^k$ above, each built from an independent Brownian motion $B_j$. The cylindrical Hermite noise on $L^2_\sigma(D)$ is the formal series
\[
Z_H^k(t) := \sum_{j\ge1} Z_{H,j}^{(k)}(t)\, e_j.
\]
\end{defn}

\begin{prop}[Properties of the cylindrical noise]\label{prop:cylindrical-props}
\begin{enumerate}[label=(\roman*)]
\item \emph{Cylindrical covariance.} For every $\varphi,\psi\in L^2_\sigma(D)$, writing $\langle Z_H^k(t),\varphi\rangle := \sum_j \varphi_j Z_{H,j}^{(k)}(t)$, $\varphi_j:=\langle\varphi,e_j\rangle$ (a well-defined real-valued process, convergent in $L^2(\Omega)$ since $\sum_j\varphi_j^2<\infty$), one has
\[
\E\bigl[\langle Z_H^k(t),\varphi\rangle\langle Z_H^k(s),\psi\rangle\bigr] = \langle\varphi,\psi\rangle\,R_H(t,s), \qquad R_H(t,s):=\tfrac12(t^{2H}+s^{2H}-|t-s|^{2H}).
\]
In particular, the spatial covariance operator of the noise is the identity: $Z_H^k$ is spatially white.
\item \emph{Finite-dimensional truncations.} For every $N\ge1$, $Z_H^{k,N}(t):=\sum_{j\le N}Z_{H,j}^{(k)}(t)e_j$ is a genuine $L^2_\sigma(D)$-valued process, $H$-self-similar with stationary increments (coordinatewise, by Proposition~\ref{prop:hermite-props}), but $\E\|Z_H^{k,N}(t)\|^2_{L^2_\sigma} = N\cdot t^{2H} \to \infty$ as $N\to\infty$ for every $t>0$: the series in Definition~\ref{def:hermite-cylindrical} does not converge in $L^2_\sigma(D)$, consistent with (i).
\end{enumerate}
\end{prop}

\begin{proof}
Part (i) follows from independence across modes (cross-terms $j\ne j'$ vanish) and Proposition~\ref{prop:hermite-props}(i) applied to each mode: $\E[\langle Z_H^k(t),\varphi\rangle\langle Z_H^k(s),\psi\rangle] = \sum_j\varphi_j\psi_j\,\E[Z_{H,j}^{(k)}(t)Z_{H,j}^{(k)}(s)] = R_H(t,s)\sum_j\varphi_j\psi_j = R_H(t,s)\langle\varphi,\psi\rangle$. Part (ii) follows similarly, using $\E[Z_{H,j}^{(k)}(t)^2]=t^{2H}$ for each of the $N$ independent modes.
\end{proof}

Definition~\ref{def:hermite-cylindrical} does not by itself define a stochastic integral against $Z_H^k$ for operator-valued integrands; this is carried out in general in Section~\ref{sec:wiener}, and applied to the stochastic convolution in Section~\ref{sec:convolution}.

\begin{rmk}
This construction is the natural generalization, to arbitrary Hermite order $k$, of the $Q$-colored fractional Brownian motion used by Zou, Lv and Wu \cite[Section 2]{zou2018stochastic} for $k=1$, and is compatible with the stochastic integration theory for Hermite-driven random fields developed in \cite{CoupekDuncanPasikDuncan2022,CoupekKrizSvoboda2025}. Unlike \cite{zou2018stochastic}, who require a trace-class spatial covariance operator $Q$, Proposition~\ref{prop:cylindrical-props}(i) shows the present construction takes $Q=\mathrm{Id}$: the noise is cylindrical (spatially white) and non-trace-class, and the summability needed to make the stochastic convolution well-defined in $\dot H^\nu$ (Section~\ref{sec:convolution}) comes entirely from the smoothing of the Mittag-Leffler operator $\cS_\al$, not from any decay assumption on the spatial covariance of the noise itself.
\end{rmk}

\section{Wiener integral with respect to the cylindrical Hermite process}\label{sec:wiener}

In this section we construct the Wiener integral with respect to a single mode $Z_{H,j}^{(k)}$ of the cylindrical noise and prove the $L^p$ estimate announced in Theorem~\ref{thm:hyper}. Since every mode has the same law as the scalar process $Z_H^k$ of Section~\ref{subsec:hermite}, it suffices to work with a single copy, dropping the mode index $j$ for the remainder of this section.

\subsection{Definition and isometry}

Let $\cE$ be the set of elementary functions of the form
\[
f(u) = \sum_{l=1}^n a_l \mathbf{1}_{(t_l, t_{l+1}]}(u), \qquad t_1 < \cdots < t_{n+1},\ a_l \in \R.
\]
For such an $f$, we define the Wiener integral by
\[
\int_{\R} f(u)\,dZ_H^k(u) := \sum_{l=1}^n a_l \bigl( Z_H^k(t_{l+1}) - Z_H^k(t_l) \bigr).
\]
Using the representation of $Z_H^k$ as a $k$-fold Wiener-It\^o integral (Section~\ref{subsec:hermite}), one can show that
\[
\int_{\R} f(u)\,dZ_H^k(u) = I_k(F_f),
\]
where $I_k$ denotes the $k$-fold Wiener-It\^o integral and $F_f$ is a symmetric kernel derived from $f$; the details can be found in \cite[Section 3]{maejima2007wiener}.

\begin{lemma}[Isometry]\label{lem:isometry}
For any $f\in\cE$,
\[
\E\left[ \left( \int_{\R} f(u)\,dZ_H^k(u) \right)^2 \right] = \|f\|_{\cH}^2,
\]
with $\|\cdot\|_\cH$ as in Section~\ref{subsec:hermite}.
\end{lemma}

\begin{proof}
\emph{Step 1: covariance of increments.} By Proposition~\ref{prop:hermite-props}(i), $Z_H^k$ has covariance $R_H(t,s)=\frac12(t^{2H}+s^{2H}-|t-s|^{2H})$, \emph{independently of $k$}. For $t\ne s$, differentiating twice,
\[
\frac{\partial^2 R_H}{\partial t\,\partial s}(t,s) = H(2H-1)\,|t-s|^{2H-2},
\]
as a direct computation shows (for $t>s$, $\partial_tR_H=H t^{2H-1}-H(t-s)^{2H-1}$, and differentiating in $s$ gives $H(2H-1)(t-s)^{2H-2}$; the case $t<s$ follows by symmetry). Consequently, for any $a<b$ and $c<d$,
\begin{equation}
\label{eq:increment-cov}
\E\bigl[(Z_H^k(b)-Z_H^k(a))(Z_H^k(d)-Z_H^k(c))\bigr] = H(2H-1)\int_a^b\!\!\int_c^d |u-v|^{2H-2}\,du\,dv,
\end{equation}
the double integral being absolutely convergent since $2H-2>-1$.

\emph{Step 2: elementary functions.} For $f=\sum_{l=1}^n a_l\mathbf 1_{(t_l,t_{l+1}]}\in\cE$, bilinearity and \eqref{eq:increment-cov} give
\begin{align*}
\E\Bigl[\Bigl(\int_\R f\,dZ_H^k\Bigr)^2\Bigr]
&= \sum_{l,m=1}^n a_la_m\,\E\bigl[\Delta_lZ_H^k\,\Delta_mZ_H^k\bigr]\\
&= H(2H-1)\sum_{l,m=1}^n a_la_m\int_{t_l}^{t_{l+1}}\!\!\int_{t_m}^{t_{m+1}}|u-v|^{2H-2}\,du\,dv,
\end{align*}
where $\Delta_lZ_H^k:=Z_H^k(t_{l+1})-Z_H^k(t_l)$. Recombining the double sum into a single double integral, the right-hand side is exactly
\[
H(2H-1)\int_\R\int_\R f(u)f(v)\,|u-v|^{2H-2}\,du\,dv = \|f\|^2_\cH. \qedhere
\]
\end{proof}

Note that this computation uses only the covariance $R_H$, which by Proposition~\ref{prop:hermite-props}(i) is the same for every Hermite order $k$; the space $\cH$ and the isometry are therefore independent of $k$, even though the law of $Z_H^k$ is not. By this isometry, the map $f\mapsto\int_\R f\,dZ_H^k$ extends uniquely from $\cE$ to a linear isometry from $\cH$ into $L^2(\Omega)$; for any $f\in\cH$, the Wiener integral is thus well defined and satisfies the isometry property. For $k=1$, this recovers the classical Wiener integral with respect to fractional Brownian motion (see \cite[Proposition 2.1]{pipiras2000convergence}).

\subsection{$L^p$ estimates via hypercontractivity}

\begin{lemma}[$L^p$ estimates]\label{lem:hyper_estimate}
For any deterministic function $f\in\cH$ and any $p\ge2$,
\[
\left\| \int_{\R} f(u)\,dZ_H^k(u) \right\|_{L^p(\Omega)} \le (p-1)^{k/2} \|f\|_{\cH}.
\]
\end{lemma}

\begin{proof}
The random variable $\int_\R f\,dZ_H^k$ belongs to the $k$-th Wiener chaos, being (by density) again a $k$-fold Wiener-It\^o integral. Hypercontractivity (Lemma~\ref{lem:hyper}) gives $\|\int_\R f\,dZ_H^k\|_{L^p(\Omega)} \le (p-1)^{k/2}\|\int_\R f\,dZ_H^k\|_{L^2(\Omega)}$, and the $L^2$ norm equals $\|f\|_\cH$ by Lemma~\ref{lem:isometry}.
\end{proof}

This is Theorem~\ref{thm:hyper}. The restriction to deterministic integrands is essential: if the integrand were itself random, the resulting random variable would not necessarily belong to a fixed Wiener chaos, and hypercontractivity would not apply directly. This is why the driving noise enters \eqref{eq:main} additively rather than multiplicatively.

\begin{rmk}

The factor $(p-1)^{k/2}$ is sharp up to constants. For $k=1$, it is the classical Gaussian estimate $\|F\|_{L^p}\le\sqrt{p-1}\,\|F\|_{L^2}$; for $k=2$ it becomes $p-1$, characteristic of Rosenblatt-type distributions. This dependence on $k$ is the mechanism, made precise mode-by-mode in Section~\ref{sec:convolution}, behind the intermittency rate of Theorem~\ref{thm:moment}.

\end{rmk}

We record, for use in Section~\ref{sec:convolution}, that Lemma~\ref{lem:hyper_estimate} applies mode-by-mode to the stochastic convolution $Z_k$: for each fixed $j$ and $t$, the eigenvalue function $s\mapsto s_{\al,j}(t-s)$ is a genuine \emph{scalar} function of $s$, and once it is shown to belong to $\cH$ (Section~\ref{sec:convolution}), Lemma~\ref{lem:hyper_estimate} gives an $L^p(\Omega)$ bound on the $j$-th coordinate of $Z_k(t)$ directly. We now combine these mode-by-mode bounds into a single, vector-valued theory.

\subsection{Vector-valued integrals against the cylindrical noise}

Let $K$ be a real separable Hilbert space (in the applications below, $K=\dot H^\nu$). The scalar construction above extends verbatim to $K$-valued deterministic integrands: for elementary $g:\R\to K$, $g(u)=\sum_l a_l\mathbf 1_{(t_l,t_{l+1}]}(u)$ with $a_l\in K$, set $\int_\R g(u)\,dZ_H^k(u) := \sum_l a_l(Z_H^k(t_{l+1})-Z_H^k(t_l))$, a $K$-valued random variable; the isometry becomes
\[
\E\Bigl[\Bigl\|\int_\R g(u)\,dZ_H^k(u)\Bigr\|_K^2\Bigr] = \|g\|_{\cH(K)}^2,
\]
where
\[
\|g\|^2_{\cH(K)} := H(2H-1)\int_\R\!\!\int_\R \langle g(u),g(v)\rangle_K\,|u-v|^{2H-2}\,du\,dv,
\]
and extends by density to every $g$ with $\|g\|_{\cH(K)}<\infty$, exactly as in Section~\ref{subsec:hermite}. This is the direct $K$-valued analogue of Maejima and Tudor's construction \cite{maejima2007wiener}.

We now integrate against the full \emph{cylindrical} noise $Z_H^k$ of Definition~\ref{def:hermite-cylindrical}. Let $\Phi:[0,T]\to\cL_2(L^2_\sigma,K)$ be a deterministic function valued in the Hilbert-Schmidt operators from $L^2_\sigma(D)$ to $K$, such that
\begin{equation}
\label{eq:phi-integrability}
\int_0^T\int_0^T \langle\Phi(u),\Phi(v)\rangle_{\cL_2(L^2_\sigma,K)}\,|u-v|^{2H-2}\,du\,dv < \infty.
\end{equation}
For each $j$, $\Phi(\cdot)e_j:[0,T]\to K$ is a deterministic, $K$-valued function, and $\int_0^T\Phi(s)e_j\,dZ_{H,j}^{(k)}(s)$ is well defined by the paragraph above, applied to the $j$-th mode. We define
\begin{equation}
\label{eq:vector-integral}
\int_0^T \Phi(s)\,dZ_H^k(s) := \sum_{j\ge1} \int_0^T \Phi(s)e_j\,dZ_{H,j}^{(k)}(s),
\end{equation}
a series of independent $K$-valued random variables (independent across $j$, since the modes $Z_{H,j}^{(k)}$ are).

\begin{lemma}[Isometry for vector-valued integrands]\label{lem:isometry_Zk}
Let $\Phi:[0,T]\to\cL_2(L^2_\sigma,K)$ be deterministic and satisfy \eqref{eq:phi-integrability}. Then the series \eqref{eq:vector-integral} converges in $L^2(\Omega;K)$, belongs to the $k$-th Wiener chaos, and
\[
\E\Bigl[\Bigl\|\int_0^T\Phi(s)\,dZ_H^k(s)\Bigr\|_K^2\Bigr] = H(2H-1)\int_0^T\int_0^T \langle\Phi(u),\Phi(v)\rangle_{\cL_2(L^2_\sigma,K)}\,|u-v|^{2H-2}\,du\,dv.
\]
\end{lemma}

\begin{proof}
By independence across $j$, the cross-terms $j\ne j'$ vanish in the second moment of \eqref{eq:vector-integral}, so
\[
\E\Bigl[\Bigl\|\int_0^T\Phi\,dZ_H^k\Bigr\|_K^2\Bigr] = \sum_{j\ge1} \E\Bigl[\Bigl\|\int_0^T\Phi(s)e_j\,dZ_{H,j}^{(k)}(s)\Bigr\|_K^2\Bigr] = \sum_{j\ge1} \|\Phi(\cdot)e_j\|_{\cH(K)}^2
\]
by the $K$-valued isometry above, applied mode-by-mode. Expanding,
\[
\sum_{j\ge1}\|\Phi(\cdot)e_j\|^2_{\cH(K)} = H(2H-1)\int_0^T\int_0^T \Bigl(\sum_{j\ge1}\langle\Phi(u)e_j,\Phi(v)e_j\rangle_K\Bigr)|u-v|^{2H-2}\,du\,dv,
\]
and $\sum_j\langle\Phi(u)e_j,\Phi(v)e_j\rangle_K = \langle\Phi(u),\Phi(v)\rangle_{\cL_2(L^2_\sigma,K)}$ is exactly the Hilbert-Schmidt inner product of $\Phi(u)$ and $\Phi(v)$, evaluated in the orthonormal basis $\{e_j\}$. This gives the stated formula, and \eqref{eq:phi-integrability} guarantees it is finite, hence that the series \eqref{eq:vector-integral} converges in $L^2(\Omega;K)$. Each summand in \eqref{eq:vector-integral} lies in the $k$-th Wiener chaos generated by the underlying Brownian motion $B_j$, and the sum of independent chaos-$k$ elements (across the independent underlying fields $B_j$) again lies in the $k$-th Wiener chaos of the combined field.
\end{proof}

\begin{lemma}[$L^p$ estimates, vector-valued]\label{lem:hyper_vector}
Under the hypotheses of Lemma~\ref{lem:isometry_Zk}, for any $p\ge2$,
\[
\Bigl\|\int_0^T\Phi(s)\,dZ_H^k(s)\Bigr\|_{L^p(\Omega;K)} \le (p-1)^{k/2}\,\|\Phi\|_{\cH(\cL_2)},
\]
where, as in Lemma~\ref{lem:isometry_Zk},
\[
\|\Phi\|_{\cH(\cL_2)}^2 := H(2H-1)\int_0^T\!\!\int_0^T \langle\Phi(u),\Phi(v)\rangle_{\cL_2(L^2_\sigma,K)}\,|u-v|^{2H-2}\,du\,dv.
\]
\end{lemma}

\begin{proof}
Write $X:=\int_0^T\Phi(s)\,dZ_H^k(s)$, and let $\{T_r\}_{r\in[0,1]}$ be the Ornstein-Uhlenbeck semigroup of the underlying Gaussian field. Two facts combine. First, Nelson's hypercontractivity theorem gives $\|T_rg\|_{L^p(\Omega)}\le\|g\|_{L^2(\Omega)}$ for every real $g\in L^2(\Omega)$, provided $r=r(p):=(p-1)^{-1/2}$. Second, $T_r$ acts on the $k$-th chaos as multiplication by $r^k$.

The extension to $K$-valued variables is immediate from the fact that $T_r$ is given by integration against the Mehler kernel, a probability measure: for $K$-valued $F$, the vector-valued triangle inequality (Jensen) gives $\|T_rF\|_K\le T_r(\|F\|_K)$ pointwise. Applying scalar hypercontractivity to the nonnegative function $\|F\|_K\in L^2(\Omega)$,
\[
\|T_rF\|_{L^p(\Omega;K)} \le \|T_r(\|F\|_K)\|_{L^p(\Omega)} \le \|\,\|F\|_K\,\|_{L^2(\Omega)} = \|F\|_{L^2(\Omega;K)}.
\]
Since $X$ lies in the $k$-th chaos (Lemma~\ref{lem:isometry_Zk}), expanding $X=\sum_iX_if_i$ in an orthonormal basis $\{f_i\}$ of $K$ shows each $X_i$ is a real chaos-$k$ variable, so $T_rX=\sum_i(T_rX_i)f_i=r^kX$. Combining, $r^k\|X\|_{L^p(\Omega;K)}\le\|X\|_{L^2(\Omega;K)}$, i.e., $\|X\|_{L^p(\Omega;K)}\le(p-1)^{k/2}\|X\|_{L^2(\Omega;K)}$. The isometry of Lemma~\ref{lem:isometry_Zk} for $\|X\|_{L^2(\Omega;K)}$ gives the stated bound.
\end{proof}

Lemma~\ref{lem:hyper_vector}, applied with $\Phi(s)=\cS_\al(t-s)$ and $K=\dot H^\nu$, is the estimate used throughout Section~\ref{sec:convolution} to control the stochastic convolution $Z_k$.

\section{Stochastic convolution and its regularity}\label{sec:convolution}

The isometry and $L^p$ estimates for vector-valued integrands against the cylindrical noise, established in Section~\ref{sec:wiener} (Lemmas~\ref{lem:isometry_Zk}--\ref{lem:hyper_vector}), reduce the regularity theory of the stochastic convolution
\[
Z_k(t) = \int_0^t \cS_\al(t-s)\,dZ_H^k(s)
\]
to a purely deterministic question: bounding the Hilbert-Schmidt norm $\|\cS_\al(r)\|_{\cL_2(L^2_\sigma,\dot H^\nu)}$ and its time-increments. This is the content of the present section.

\subsection{Hilbert-Schmidt estimates}

For the $L^2$ estimate of the stochastic convolution, we need a more precise bound than the operator-norm estimate of Lemma~\ref{lem:MittagLeffler_estimates}: the Hilbert-Schmidt norm of $\cS_\al(r)$ from $L^2_\sigma$ to $\dot H^\nu$. For a linear operator $K$ between separable Hilbert spaces, recall $\|K\|_{\cL_2(U,V)}^2 = \sum_j\|Ke_j\|_V^2$ for any orthonormal basis $\{e_j\}$ of $U$; we take $U=L^2_\sigma(D)$, with $\{e_j\}$ the Stokes eigenbasis, and $V=\dot H^\nu$.

\begin{lemma}[Hilbert-Schmidt estimate]\label{lem:HS}
Assume $d=2$ and $0\le\nu<1$. Then for every $T>0$ there exists $C=C(\al,\nu,T,D)$ such that, for all $0<r\le T$,
\[
\|\cS_\al(r)\|_{\cL_2(L^2_\sigma,\dot H^\nu)}^2 \le C\,r^{\al(1-\nu)-2}.
\]
\end{lemma}

\begin{proof}
\emph{Step 1: reduction to a spectral sum.} The operator $\cS_\al(r)$ is diagonal in the Stokes eigenbasis with eigenvalues $s_{\al,j}(r)=r^{\al-1}E_{\al,\al}(-\lambda_j r^\al)$, so $\cS_\al(r)e_j=s_{\al,j}(r)e_j$ and $\|\cS_\al(r)e_j\|_\nu^2=\lambda_j^\nu|s_{\al,j}(r)|^2$. Using the standard Mittag-Leffler decay $|E_{\al,\al}(-x)|\le C(1+x)^{-1}$ for $x\ge0$ (see \cite[Theorem 1.6]{podlubny1999fractional}),
\[
\|\cS_\al(r)\|^2_{\cL_2} = \sum_{j\ge1} \lambda_j^\nu|s_{\al,j}(r)|^2 \le C\,r^{2\al-2}\,\Sigma(\eps), \qquad \Sigma(\eps):=\sum_{j\ge1}\frac{\lambda_j^\nu}{(1+\lambda_j r^\al)^2},
\]
and by Weyl's asymptotics $\lambda_j\sim C_2j$ there are $0<c_1\le c_2$ with $c_1 j\le\lambda_j\le c_2 j$ for all $j$, whence, writing $\eps:=c_1r^\al$,
\[
\Sigma(\eps) \le c_2^\nu\sum_{j\ge1} g(j), \qquad g(x):=\frac{x^\nu}{(1+\eps x)^2}.
\]

\emph{Step 2: sum-integral comparison for a unimodal summand.} The function $g$ is \emph{not} monotone: differentiating,
\[
g'(x) = \frac{x^{\nu-1}}{(1+\eps x)^3}\bigl[\nu+\eps x(\nu-2)\bigr],
\]
which vanishes at $x^\ast=\frac{\nu}{\eps(2-\nu)}$ (using $\nu<2$), is positive on $(0,x^\ast)$ and negative on $(x^\ast,\infty)$. Thus $g$ is unimodal, and the elementary comparison for unimodal functions applies:
\[
\sum_{j\ge1}g(j) \le \int_0^\infty g(x)\,dx + \max_{x>0}g(x).
\]
(Indeed, on each of the two monotone branches the sum of $g$ over integers is bounded by the integral over the corresponding interval, plus at most one extra term, the maximum.)

\emph{Step 3: the two contributions.} Substituting $y=\eps x$,
\[
\int_0^\infty \frac{x^\nu}{(1+\eps x)^2}\,dx = \eps^{-\nu-1}\int_0^\infty\frac{y^\nu}{(1+y)^2}\,dy = \eps^{-\nu-1}B(\nu+1,1-\nu),
\]
the Beta integral converging precisely because $\nu+1>0$ and $1-\nu>0$, i.e.\ $\nu<1$. For the maximum,
\[
\max_{x>0}g(x) = g(x^\ast) = \Bigl(\tfrac{\nu}{2-\nu}\Bigr)^{\nu}\Bigl(\tfrac{2-\nu}{2}\Bigr)^{2}\eps^{-\nu} = C_\nu\,\eps^{-\nu}.
\]
Hence $\Sigma(\eps)\le C(\eps^{-\nu-1}+\eps^{-\nu}) = C\eps^{-\nu-1}(1+\eps)$.

\emph{Step 4: conclusion.} Since $\eps=c_1r^\al\le c_1T^\al$ for $r\le T$, the factor $1+\eps$ is bounded by a constant depending on $T$, so $\Sigma(\eps)\le C_T\,\eps^{-\nu-1}=C_T'\,r^{-\al(\nu+1)}$. Combining with Step 1,
\[
\|\cS_\al(r)\|^2_{\cL_2}\le C\,r^{2\al-2}\,r^{-\al(\nu+1)} = C\,r^{2\al-2-\al\nu-\al} = C\,r^{\al(1-\nu)-2}. \qedhere
\]
\end{proof}

\begin{rmk}
The condition $\nu<1$ is sharp in dimension $2$: at $\nu=1$ the Beta integral of Step~3 becomes $\int_0^\infty y(1+y)^{-2}dy$, which diverges logarithmically. This reflects that space-white noise in dimension $2$ belongs to $\dot H^{-1-\varepsilon}$ for every $\varepsilon>0$ but not to $\dot H^{-1}$, so $\cS_\al(r)$ is not Hilbert--Schmidt into $\dot H^1$; it is not related to criticality for the two-dimensional Navier--Stokes equations, which are well posed in $H^1$. The restriction is in any case forced by $\al(1-\nu)+2H>2$. The restriction $r\le T$ is not cosmetic: the maximum term of Step~3 contributes $r^{\al(2-\nu)-2}=r^\al\cdot r^{\al(1-\nu)-2}$, absorbed into the stated bound only for bounded $r$.
\end{rmk}

The increment regularity of $Z_k$ (Section~\ref{sec:convolution}, below) requires, in addition, a Hilbert-Schmidt bound on the time-derivative of $\cS_\al$.

\begin{lemma}[HS estimate for the derivative]\label{lem:HS_deriv}
Assume $d=2$ and $0\le\nu<1$. Then for every $T>0$ there exists $C=C(\al,\nu,T,D)$ such that, for all $0<r\le T$,
\[
\left\|\frac{d}{dr}\cS_\al(r)\right\|_{\cL_2(L^2_\sigma,\dot H^\nu)}^2 \le C\,r^{\al(1-\nu)-4}.
\]
\end{lemma}

\begin{proof}
Write $x:=\lambda_j r^\al$ and start from the integral representation
\[
s_{\al,j}(r)=r^{\al-1}\int_0^\infty\al\theta\,\xi_\al(\theta)\,e^{-x\theta}\,d\theta .
\]
Differentiation in $r$ under the integral sign is justified by dominated convergence, the differentiated integrand being bounded by $\al\theta^2\xi_\al(\theta)$ times a locally bounded factor, integrable by \eqref{eq:moment} at $\varrho=2$. This gives
\[
s_{\al,j}'(r) = (\al-1)r^{\al-2}\int_0^\infty\al\theta\,\xi_\al(\theta)\,e^{-x\theta}\,d\theta \;-\; \al^2\lambda_j r^{2\al-2}\int_0^\infty\theta^2\xi_\al(\theta)\,e^{-x\theta}\,d\theta.
\]

\emph{First term.} The integral is $E_{\al,\al}(-x)$, bounded by $C(1+x)^{-1}$ as in Step 1 of Lemma~\ref{lem:HS}; the term is thus $O\bigl(r^{\al-2}(1+x)^{-1}\bigr)$.

\emph{Second term.} For $x\ge1$, the elementary bound $\theta^2e^{-x\theta}\le\bigl(\sup_{s\ge0}s^2e^{-s}\bigr)x^{-2}=4e^{-2}x^{-2}$ (substituting $s=x\theta$) gives $\int_0^\infty\theta^2\xi_\al(\theta)e^{-x\theta}d\theta\le4e^{-2}x^{-2}$, using $\int_0^\infty\xi_\al=1$. For $x\le1$, the same integral is bounded by $\int_0^\infty\theta^2\xi_\al(\theta)d\theta=\Ga(3)/\Ga(1+2\al)$ by \eqref{eq:moment} at $\varrho=2$. Since $(1+x)^{-2}$ is bounded above and below by constants times $\min(1,x^{-2})$, the two regimes combine into $\int_0^\infty\theta^2\xi_\al(\theta)e^{-x\theta}d\theta\le C(1+x)^{-2}$ for all $x>0$. Hence the second term is bounded by
\[
C\,\lambda_j r^{2\al-2}(1+x)^{-2} = C\,r^{\al-2}\cdot x(1+x)^{-2} \le C\,r^{\al-2}(1+x)^{-1},
\]
using $\lambda_jr^{2\al-2}=r^{\al-2}\cdot\lambda_jr^{\al}=r^{\al-2}x$ and $x(1+x)^{-2}\le(1+x)^{-1}$.

\emph{Conclusion.} Both terms are $O\bigl(r^{\al-2}(1+\lambda_jr^\al)^{-1}\bigr)$, so $|s_{\al,j}'(r)|\le Cr^{\al-2}(1+\lambda_jr^\al)^{-1}$, and squaring and summing gives exactly the spectral sum $\Sigma(\eps)$ already estimated in Steps 2--4 of Lemma~\ref{lem:HS} (same summand, hence same restriction $r\le T$):
\[
\Bigl\|\frac{d}{dr}\cS_\al(r)\Bigr\|_{\cL_2}^2 = \sum_{j\ge1}\lambda_j^\nu|s_{\al,j}'(r)|^2 \le C\,r^{2\al-4}\,\Sigma(\eps) \le C\,r^{2\al-4}\,r^{-\al(\nu+1)} = C\,r^{\al(1-\nu)-4}. \qedhere
\]
\end{proof}

\subsection{$L^2$ and $L^p$ estimates of the stochastic convolution}

\begin{lemma}[$L^2$ estimate of $Z_k$]\label{lem:Zk_L2}
Under the assumptions of Lemma~\ref{lem:HS} and the condition $\al(1-\nu)+2H>2$,
\[
\|Z_k(t)\|_{L^2(\Omega;\dot H^\nu)}^2 \le C\,t^{\al(1-\nu)+2H-2}.
\]
\end{lemma}

\begin{proof}
\emph{Step 1: verification of the integrability hypothesis.} To apply Lemma~\ref{lem:isometry_Zk} with $\Phi(s):=\cS_\al(t-s)$ we must first check the hypothesis \eqref{eq:phi-integrability}. By Cauchy-Schwarz for the Hilbert-Schmidt inner product and Lemma~\ref{lem:HS} (applicable since $t-u\le t\le T$), writing $a:=\frac{\al(1-\nu)-2}{2}$,
\[
\bigl|\langle\cS_\al(t-u),\cS_\al(t-v)\rangle_{\cL_2}\bigr| \le \|\cS_\al(t-u)\|_{\cL_2}\|\cS_\al(t-v)\|_{\cL_2} \le C(t-u)^a(t-v)^a,
\]
so it suffices to show that $\int_0^t\int_0^t(t-u)^a(t-v)^a|u-v|^{2H-2}dudv<\infty$; Steps 2--3 below establish exactly this, so the hypothesis holds and the isometry applies, giving
\begin{align*}
\|Z_k(t)\|_{L^2}^2 &= H(2H-1)\int_0^t\!\!\int_0^t \langle\cS_\al(t-u),\cS_\al(t-v)\rangle_{\cL_2}\,|u-v|^{2H-2}\,du\,dv\\
&\le C\int_0^t\!\!\int_0^t (t-u)^a(t-v)^a|u-v|^{2H-2}\,du\,dv.
\end{align*}

\emph{Step 2: scaling.} Changing variables $u=t-x$, $v=t-y$ (so $(t-u)^a=x^a$ and $|u-v|=|x-y|$), then rescaling $x=t\xi$, $y=t\eta$ with $\xi,\eta\in[0,1]$, the integrand picks up $t^a\cdot t^a\cdot t^{2H-2}$ and the two differentials $t\,d\xi\,t\,d\eta$, whence
\[
\|Z_k(t)\|_{L^2}^2 \le Ct^{2a+2H}\int_0^1\int_0^1 \xi^a\eta^a|\xi-\eta|^{2H-2}\,d\xi\,d\eta.
\]
The double integral is finite by symmetry and the substitution $\eta=\xi s$ on the region $\eta<\xi$: the inner integral gives $\int_0^\xi \eta^a(\xi-\eta)^{2H-2}\,d\eta = \xi^{a+2H-1}B(a+1,2H-1)$, using the Beta-function identity $\int_0^1(1-s)^{2H-2}s^a\,ds=B(a+1,2H-1)$, convergent since $a>-1$ (as $\al(1-\nu)>0$) and $H>1/2$; combined with the outer factor $\xi^a$, integrating the resulting $\xi^{2a+2H-1}$ over $[0,1]$ requires $2a+2H>0$, i.e., $\al(1-\nu)+2H>2$. This gives $\|Z_k(t)\|_{L^2}^2\le Ct^{2a+2H}=Ct^{\al(1-\nu)+2H-2}$.
\end{proof}

\begin{thm}[$L^p$ regularity of $Z_k$]\label{thm:reg1p}
Assume $0<\al<1$, $H\in(1/2,1)$, $0\le\nu<1$, and $\al(1-\nu)+2H>2$. Then for any $p\ge2$ and $t\in[0,T]$,
\[
\|Z_k(t)\|_{L^p(\Omega;\dot H^\nu)} \le C_{p,k}\,t^{\frac{\al(1-\nu)+2H-2}{2}}, \qquad C_{p,k}=(p-1)^{k/2}C_0.
\]
\end{thm}

\begin{proof}
By Lemma~\ref{lem:isometry_Zk} applied with $\Phi(s)=\cS_\al(t-s)$ and $K=\dot H^\nu$ (whose hypothesis was verified in Step 1 of Lemma~\ref{lem:Zk_L2}), $Z_k(t)$ is a $\dot H^\nu$-valued element of the $k$-th Wiener chaos. The vector-valued hypercontractivity estimate of Lemma~\ref{lem:hyper_vector} therefore applies and gives
\[
\|Z_k(t)\|_{L^p(\Omega;\dot H^\nu)} \le (p-1)^{k/2}\,\|Z_k(t)\|_{L^2(\Omega;\dot H^\nu)},
\]
and Lemma~\ref{lem:Zk_L2} bounds the right-hand side by $C_0\,t^{\frac{\al(1-\nu)+2H-2}{2}}$, whence the claim with $C_{p,k}=(p-1)^{k/2}C_0$.
\end{proof}

\begin{lemma}[Supremum moments of the stochastic convolution]\label{lem:Zk-supmom}
Assume the hypotheses of Theorems~\ref{thm:reg1p} and~\ref{thm:reg2p} with $\nu=\nu_q:=1-\frac2q\in(0,1)$. Then for every $p\ge1$ and $T>0$, $\E\bigl[\sup_{t\le T}\|Z_k(t)\|_{L^q_\sigma}^p\bigr]<\infty$. Moreover there are $p_0$ and $C_T$ with $\bigl\|\sup_{t\le T}\|Z_k(t)\|_{L^q_\sigma}\bigr\|_{L^p(\Omega)}\le C_T(p-1)^{k/2}$ for $p\ge p_0$.
\end{lemma}
\begin{proof}
\emph{Step 1.} Put $\gamma_q:=\frac\al q+H-1>0$. By Theorem~\ref{thm:reg2p} with $\nu=\nu_q$ and the embedding $\dot H^{\nu_q}\hookrightarrow L^q_\sigma$ ($\frac1{p(\nu_q)}=\frac12-\frac{\nu_q}2=\frac1q$), $\|Z_k(t)-Z_k(s)\|_{L^p(\Omega;L^q_\sigma)}\le C_0(p-1)^{k/2}|t-s|^{\gamma_q}$. By Kolmogorov--Chentsov with some $p>1/\gamma_q$, $Z_k$ has a continuous version in $L^q_\sigma$; we use it from now on.

\emph{Step 2 (Garsia--Rodemich--Rumsey).} Fix $0<\gamma'<\gamma_q$, $\delta:=\gamma_q-\gamma'$, and let $p\ge p_0:=2/\delta$. Take $\Psi(x)=x^p$, $\varrho(u)=u^{\gamma'+2/p}$ and $B:=\int_0^T\!\!\int_0^T\bigl(\|Z_k(t)-Z_k(s)\|/\varrho(|t-s|)\bigr)^p\,ds\,dt$. The GRR lemma (Nualart, Lemma A.3.1) gives
\[
\|Z_k(t)-Z_k(s)\|\le 8\int_0^{|t-s|}\Bigl(\frac{4B}{u^2}\Bigr)^{1/p}d\varrho(u)=8(4B)^{1/p}\Bigl(1+\frac2{p\gamma'}\Bigr)|t-s|^{\gamma'}.
\]
By Step 1, $\E B\le C_0^p(p-1)^{kp/2}\iint|t-s|^{p\delta-2}\,ds\,dt=C_0^p(p-1)^{kp/2}\,\frac{2T^{p\delta}}{(p\delta-1)\,p\delta}$, which is finite since $p\delta>1$ and gives $(\E B)^{1/p}\le C_T(p-1)^{k/2}$ uniformly in $p\ge p_0$.

\emph{Step 3.} Since $Z_k(0)=0$, $\sup_t\|Z_k(t)\|\le T^{\gamma'}\cdot8(4B)^{1/p}(1+\frac2{p\gamma'})$, and $\|B^{1/p}\|_{L^p(\Omega)}=(\E B)^{1/p}$. This proves the quantitative bound.

\emph{Step 4.} For $1\le p<p_0$, use $\|\cdot\|_{L^p(\Omega)}\le\|\cdot\|_{L^{p_0}(\Omega)}$.
\end{proof}

\subsection{Increment regularity}

\begin{thm}[Increment regularity of $Z_k$]\label{thm:reg2p}
Under the assumptions of Theorem~\ref{thm:reg1p}, for any $0\le t_1<t_2\le T$ and $p\ge2$,
\[
\|Z_k(t_2)-Z_k(t_1)\|_{L^p(\Omega;\dot H^\nu)} \le C_{p,k}\,(t_2-t_1)^{\gamma}, \qquad \gamma = \frac{\al(1-\nu)+2H-2}{2}.
\]
\end{thm}

\begin{proof}
Write $Z_k(t_2)-Z_k(t_1)=I_1+I_2$,
\[
I_1 = \int_0^{t_1}[\cS_\al(t_2-s)-\cS_\al(t_1-s)]\,dZ_H^k(s), \qquad I_2=\int_{t_1}^{t_2}\cS_\al(t_2-s)\,dZ_H^k(s).
\]
By stationarity of the increments of $Z_H^k$, $I_2$ has the same distribution as $Z_k(t_2-t_1)$, so $\|I_2\|_{L^p}\le C_{p,k}(t_2-t_1)^{\gamma}$ by Theorem~\ref{thm:reg1p}.

For $I_1$, write $\tau:=t_2-t_1$, $\Delta(s):=\cS_\al(t_2-s)-\cS_\al(t_1-s)$ for $s\in[0,t_1]$, $\|\cdot\|:=\|\cdot\|_{\cL_2(L^2_\sigma,\dot H^\nu)}$, and $b:=\frac{\al(1-\nu)-2}{2}$. Split $[0,t_1]$ into a \emph{near} region $N:=\{s:t_1-s<\tau\}$ and a \emph{far} region $F:=\{s:t_1-s\ge\tau\}$, and decompose $\Delta=\Delta\mathbf 1_N+\Delta\mathbf 1_F$ accordingly.

Since $H>1/2$, the kernel $|u-v|^{2H-2}$ is positive definite, so by Lemma~\ref{lem:isometry_Zk} the quantity $\Phi\mapsto\bigl(H(2H-1)\iint\langle\Phi(u),\Phi(v)\rangle_{\cL_2}|u-v|^{2H-2}dudv\bigr)^{1/2}$ is a genuine Hilbert norm, which we denote $\|\Phi\|_{\cH(\cL_2)}$. The triangle inequality for this norm gives
\begin{equation}
\label{eq:nearfar-split}
\|I_1\|^2_{L^2(\Omega;\dot H^\nu)} = \|\Delta\|^2_{\cH(\cL_2)} \le 2\bigl(\|\Delta\mathbf 1_N\|^2_{\cH(\cL_2)}+\|\Delta\mathbf 1_F\|^2_{\cH(\cL_2)}\bigr),
\end{equation}
so it suffices to bound the two diagonal contributions separately; no cross term $N\times F$ needs to be estimated.

\emph{Near region.} By the triangle inequality and Lemma~\ref{lem:HS}, $\|\Delta(s)\|\le\|\cS_\al(t_2-s)\|+\|\cS_\al(t_1-s)\|\le C(t_1-s)^b$ for $t_1-s<\tau\le t_2-s$. Applying Cauchy-Schwarz to $\langle\Delta(u),\Delta(v)\rangle_{\cL_2}$ and repeating the computation of Lemma~\ref{lem:Zk_L2} on the interval of length $\tau$,
\[
\|\Delta\mathbf 1_N\|^2_{\cH(\cL_2)} \le C\int_{t_1-\tau}^{t_1}\int_{t_1-\tau}^{t_1}(t_1-u)^b(t_1-v)^b|u-v|^{2H-2}\,du\,dv \le C\,\tau^{2b+2H} = C\,\tau^{\al(1-\nu)+2H-2}.
\]

\emph{Far region.} By the mean value theorem and Lemma~\ref{lem:HS_deriv}, for $t_1-s\ge\tau$,
\[
\|\Delta(s)\| \le \tau\sup_{r\in[t_1-s,t_2-s]}\Bigl\|\frac{d}{dr}\cS_\al(r)\Bigr\| \le C\,\tau\,(t_1-s)^{b-1},
\]
using that $r\mapsto r^{b-1}$ is decreasing (as $b-1<0$), so the supremum is attained at $r=t_1-s$. Substituting $x=t_1-u,\ y=t_1-v\in[\tau,t_1]$ and rescaling $x=\tau X,\ y=\tau Y$,
\begin{align*}
\|\Delta\mathbf 1_F\|^2_{\cH(\cL_2)}
&\le C\tau^2\int_\tau^{t_1}\!\!\int_\tau^{t_1} x^{b-1}y^{b-1}|x-y|^{2H-2}\,dx\,dy\\
&= C\,\tau^{2b+2H}\int_1^{t_1/\tau}\!\!\int_1^{t_1/\tau} X^{b-1}Y^{b-1}|X-Y|^{2H-2}\,dX\,dY.
\end{align*}
The remaining double integral is bounded uniformly in $t_1/\tau$ by a constant $K_\infty$. Indeed, by symmetry it suffices to consider $Y<X$ and to split according to whether $Y\ge X/2$ or $Y<X/2$. On $\{X/2\le Y<X\}$ we have $Y^{b-1}\le 2^{1-b}X^{b-1}$ and $\int_{X/2}^X|X-Y|^{2H-2}dY=\frac{(X/2)^{2H-1}}{2H-1}$, contributing $CX^{2b+2H-3}$, integrable on $[1,\infty)$ since $2b+2H-3<-1\iff b+H<1$, which holds as $b<0<1-H$. On $\{1\le Y<X/2\}$ we have $|X-Y|\ge X/2$, so the contribution is at most $CX^{b-1}X^{2H-2}\int_1^\infty Y^{b-1}dY=C'X^{b+2H-3}$, integrable since $b+2H-3<-1\iff b+2H<2$, which holds as $b<0$ and $H<1$. (Both integrals over $Y$ converge because $b-1<-1$ and $2H-2>-1$.) Hence
\[
\|\Delta\mathbf 1_F\|^2_{\cH(\cL_2)} \le C\,K_\infty\,\tau^{2b+2H} = C\,\tau^{\al(1-\nu)+2H-2}.
\]

Both regions therefore contribute the same order, so by \eqref{eq:nearfar-split},
\[
\|I_1\|_{L^2(\Omega;\dot H^\nu)}^2\le C(t_2-t_1)^{\al(1-\nu)+2H-2},
\]
and hypercontractivity gives $\|I_1\|_{L^p}\le C_{p,k}(t_2-t_1)^{\gamma}$. Combining with the estimate for $I_2$ completes the proof.
\end{proof}

Together, Theorems~\ref{thm:reg1p} and~\ref{thm:reg2p} constitute Theorem~\ref{thm:Zk}. We record separately, as promised in Section~\ref{sec:prelim} (Proposition~\ref{prop:endpoint-fact}), the underlying reason the near/far decomposition above is needed rather than a direct difference estimate.

\begin{prop}\label{prop:endpoint-fact-full}
For $0<\al<1$ and $\nu\ge0$, no bound of the form $\|[\cS_\al(t_2)-\cS_\al(t_1)]\chi\|_{\dot H^\nu}\le C(t_2-t_1)^\theta\|\chi\|$ can hold for all $0\le t_1<t_2\le T$ with $C$ independent of $t_1$.
\end{prop}

\begin{proof}
Fix $t_2>0$ and an eigenvector $e_j$ of the Stokes operator. Since $E_{\al,\al}$ is entire with $E_{\al,\al}(0)=1/\Gamma(\al)>0$,
\[
\|\cS_\al(t_1)e_j\|_{\dot H^\nu} \sim \frac{\lambda_j^{\nu/2}}{\Gamma(\al)}\,t_1^{\al-1} \longrightarrow \infty \qquad (t_1\to0^+),
\]
while $\|\cS_\al(t_2)e_j\|_{\dot H^\nu}$ stays fixed. The reverse triangle inequality gives $\|[\cS_\al(t_2)-\cS_\al(t_1)]\|_{\cL(\dot H^0,\dot H^\nu)}\ge\|\cS_\al(t_1)\|_{\cL(\dot H^0,\dot H^\nu)}-\|\cS_\al(t_2)\|_{\cL(\dot H^0,\dot H^\nu)}\to\infty$, contradicting any $t_1$-independent bound.
\end{proof}

This is why the proof of Theorem~\ref{thm:reg2p} above cannot bound $\|\Delta(s)\|$ by a single, $s$-independent power of $\tau$ over the whole range $s\in[0,t_1]$: near $s=t_1$, $\Delta(s)$ is genuinely dominated by the individually diverging term $\cS_\al(t_1-s)$, which is precisely why the near region is treated by the crude triangle-inequality bound of Lemma~\ref{lem:HS} rather than by a mean-value estimate.

\section{Local existence via a Da Prato-Debussche decomposition}\label{sec:existence}

We fix, throughout this section, an admissible Lebesgue exponent $q$ satisfying $2<q<\al/(1-H)$, guaranteed to exist by the standing condition $\al+2H>2$.

\begin{assumption}\label{ass:f}
The function $f:L^q_\sigma\to L^q_\sigma$ is  a deterministic, continuous map that satisfies for all $u,v\in L^q_\sigma$,
\[
\|f(u)-f(v)\|_{L^q_\sigma} \le L_f\|u-v\|_{L^q_\sigma}, \qquad \|f(u)\|_{L^q_\sigma} \le L_f(1+\|u\|_{L^q_\sigma}),
\]
for some constant $L_f>0$.
\end{assumption}

\subsection{The decomposition}

Set $v(t):=u(t)-Z_k(t)$. Substituting into the mild formulation \eqref{eq:mild}, $v$ solves
\begin{equation}
\label{eq:v-mild}
v(t) = E_\al(t)u_0 + \int_0^t \cS_\al(t-s)\bigl[B(v(s)+Z_k(s))+f(v(s)+Z_k(s))\bigr]\,ds.
\end{equation}
We look for $v$, not $u$, in $C([0,T];L^q_\sigma)$: since $u_0\in L^q_\sigma$ and $Z_k(t)\in L^q_\sigma$ (Theorem~\ref{thm:reg1p}, via the Sobolev embedding $\dot H^{\nu_q}\hookrightarrow L^q_\sigma$, $\nu_q=1-\frac2q<1$), it suffices to solve \eqref{eq:v-mild} for $v$ with the noise entering only through the already-controlled quantity $Z_k(s)$, rather than through a fixed point for $u$ itself, for which the bilinear term $B(u)$ would require $u\in\dot H^{\nu_q+1}$ -- a regularity $Z_k$ never attains (Section~\ref{sec:convolution}).

\subsection{The bilinear estimate}

We first record, for completeness and for later reference, a bilinear estimate in the Hilbert scale $\dot H^\nu$, proved by duality. It is \emph{not} used in the fixed-point argument below -- that argument works entirely in the Lebesgue scale $L^q_\sigma$, via Lemma~\ref{lem:bilinear} -- but it is the natural counterpart of Proposition~\ref{prop:sobolev} in the Hilbert setting, and is the estimate one would start from in an extension to multiplicative noise (Section~\ref{sec:conclusion}), where the Hilbert structure is needed for Malliavin-calculus arguments.

\begin{lemma}[Dual bilinear estimate]\label{lem:B_estimates_v2}
For $0<\nu<1$, there exists $C>0$ such that
\[
\|B(u,v)\|_{\nu-1} \le C\,\|u\|_\nu\,\|v\|_{\nu+1}.
\]
\end{lemma}

\begin{proof}
\emph{Step 1: duality without integration by parts.} Since $\dot H^{\nu-1}$ is the dual of $\dot H^{1-\nu}$,
\[
\|B(u,v)\|_{\nu-1}=\sup_{\|\varphi\|_{1-\nu}\le1}|\langle B(u,v),\varphi\rangle|.
\]
The point of the argument is to estimate $\langle B(u,v),\varphi\rangle$ \emph{directly}, without moving a derivative onto $\varphi$: writing $B(u,v)=-P_H[(u\cdot\nabla)v]$ and using that $P_H$ is self-adjoint on $L^2_\sigma$ and $\varphi$ is divergence-free,
\[
|\langle B(u,v),\varphi\rangle| = \Bigl|\int_D(u\cdot\nabla v)\cdot\varphi\,dx\Bigr| \le \|u\|_{L^{p_1}}\,\|\nabla v\|_{L^{p_2}}\,\|\varphi\|_{L^{p_3}}
\]
by the three-factor H\"older inequality, valid as soon as $\frac1{p_1}+\frac1{p_2}+\frac1{p_3}\le1$ on the bounded domain $D$ (using $L^p(D)\subset L^{p'}(D)$ for $p\ge p'$ to absorb any slack).

\emph{Step 2: choice of exponents.} Apply the third item of Proposition~\ref{prop:sobolev}, $\dot H^s\hookrightarrow L^{p(s)}$ with $\frac1{p(s)}=\frac12-\frac s2$, valid for $0\le s<1$, to the three factors:
\[
u\in\dot H^\nu\hookrightarrow L^{p(\nu)}, \qquad \nabla v\in\dot H^{\nu}\hookrightarrow L^{p(\nu)} \ \ (\text{since }v\in\dot H^{\nu+1}), \qquad \varphi\in\dot H^{1-\nu}\hookrightarrow L^{p(1-\nu)}.
\]
The three exponents $s\in\{\nu,\nu,1-\nu\}$ must all lie in $[0,1)$: this holds precisely when $0<\nu<1$. (At the excluded endpoint $\nu=0$ one would need $s=1-\nu=1$, i.e.\ $\dot H^1\hookrightarrow L^\infty$, which fails in dimension two -- this is the critical Sobolev case; hence the strict inequality $\nu>0$ in the statement.)

\emph{Step 3: the H\"older condition.} With these choices,
\[
\frac1{p_1}+\frac1{p_2}+\frac1{p_3} = \Bigl(\frac12-\frac\nu2\Bigr)+\Bigl(\frac12-\frac\nu2\Bigr)+\Bigl(\frac12-\frac{1-\nu}2\Bigr) = 1-\frac\nu2 \le 1,
\]
so the H\"older inequality of Step 1 applies, with strict slack $\frac\nu2>0$ absorbed as indicated. Combining,
\[
|\langle B(u,v),\varphi\rangle| \le C\|u\|_\nu\,\|v\|_{\nu+1}\,\|\varphi\|_{1-\nu},
\]
and taking the supremum over $\|\varphi\|_{1-\nu}\le1$ gives the claim. Note that $\|\nabla v\|_{L^{p(\nu)}}\le C\|v\|_{\nu+1}$ uses only $\dot H^\nu\hookrightarrow L^{p(\nu)}$ applied to $\nabla v$, never $v\in W^{1,\infty}$; this is what distinguishes the present duality argument from a fractional Leibniz (Kato-Ponce) estimate at the top order, and matches the approach of \cite[Chapter III]{temam1984navier}.
\end{proof}

The estimate actually used in the fixed-point argument works directly in the Lebesgue scale, applied to the combined quantity $w:=v+Z_k$ rather than to $v$ and $Z_k$ separately.

\begin{lemma}[Calder\'on-Zygmund bilinear estimate]\label{lem:bilinear}
For $w\in L^q_\sigma$,
\[
\|A^{-1/2}P_H\,\mathrm{div}(w\otimes w)\|_{L^{q/2}_\sigma} \le C\,\|w\|_{L^q_\sigma}^2.
\]
\end{lemma}

\begin{proof}
\emph{Step 1: the divergence-form identity.} Since $w$ is divergence-free, the chain rule gives, componentwise, $[(w\cdot\nabla)w]_i=\sum_l w_l\partial_l w_i=\sum_l\partial_l(w_lw_i)-w_i\,\mathrm{div}\,w=\mathrm{div}(w\otimes w)_i$, so $B(w)=-P_H[(w\cdot\nabla)w]=-P_H[\mathrm{div}(w\otimes w)]$, the identity holding first for smooth divergence-free fields and then in the sense of distributions by density.

\emph{Step 2: boundedness of the composite operator.} The divergence operator alone is \emph{not} bounded on $L^{q/2}$: it is a genuine first-order differential operator, and no bound $\|\mathrm{div}\,F\|_{L^{q/2}}\le C\|F\|_{L^{q/2}}$ can hold. What is bounded is the \emph{composite} $A^{-1/2}P_H\,\mathrm{div}$, in which the smoothing of $A^{-1/2}$ exactly compensates the derivative: symbolically it is homogeneous of order $-1+0+1=0$, and it is realized as a matrix of Calder\'on-Zygmund singular integral operators. Its $L^{r}$-boundedness for every $1<r<\infty$, on a bounded domain with Dirichlet boundary conditions, is classical in the mild-solution theory of the Navier-Stokes equations; see \cite[Section 4]{ferrario2018lp} and the references therein. Applying this with $r=q/2$, admissible since $q>2$ gives $1<q/2<\infty$,
\[
\|A^{-1/2}P_H\,\mathrm{div}(w\otimes w)\|_{L^{q/2}_\sigma} \le C\,\|w\otimes w\|_{L^{q/2}}.
\]

\emph{Step 3: the quadratic bound.} Pointwise, $|w\otimes w|\le|w|^2$, so $\|w\otimes w\|_{L^{q/2}}^{q/2}\le\int_D|w|^{q}\,dx$, i.e.\ $\|w\otimes w\|_{L^{q/2}}\le\|w\|_{L^q}^2$. (Equivalently, this is H\"older's inequality with the conjugate pair $(2,2)$ applied to the product of the two factors $w$ in $L^q$.) Combining with Step 2 gives the claim.
\end{proof}

Writing $B(w)=A^{1/2}\cdot\bigl[-A^{-1/2}P_H\,\mathrm{div}(w\otimes w)\bigr]$ and combining Lemma~\ref{lem:bilinear} with Lemma~\ref{lem:Lq-smoothing} at $\beta=\frac12$, $r=\frac q2$ (so that $A^{1/2}\cS_\al(t-s)$ acts as a bounded operator $L^{q/2}_\sigma\to L^q_\sigma$, absorbing the order-zero factor of Lemma~\ref{lem:bilinear}),
\begin{equation}
\label{eq:bilinear-time}
\|\cS_\al(t-s)B(w(s))\|_{L^q_\sigma} \le C\,(t-s)^{\al(\frac12-\frac1q)-1}\,\|w(s)\|_{L^q_\sigma}^2,
\end{equation}
integrable in $s$ as soon as $q>2$ -- independently of $\al$. Since $f$ maps $L^q_\sigma$ to itself (Assumption~\ref{ass:f}), the corresponding bound for the forcing term uses Lemma~\ref{lem:Lq-smoothing} at $\beta=0$, $r=q$, giving the milder exponent $\al-1$ (always integrable, since $\al>0$).

\subsection{The fixed-point argument}

We adapt the argument of Ferrario and Olivera \cite[Theorem 4.1]{ferrario2018lp} to the Mittag-Leffler operator and to general Hermite order $k$.

\begin{thm}[Local existence and uniqueness]\label{thm:exist_final}
Under Assumption~\ref{ass:f} and the condition $\al+2H>2$, for any $u_0\in L^q_\sigma$, there exists an almost surely positive random time $T^*>0$ and a unique $v\in C([0,T^*];L^q_\sigma)$ solving \eqref{eq:v-mild}.
\end{thm}

\begin{proof}
Let $M\ge1$ be the constant of the uniform semigroup bound \eqref{eq:semigroup-Lp} at $\beta=0,r=q$ ($\|T(s)u\|_{L^q_\sigma}\le M\|u\|_{L^q_\sigma}$ for every $s\ge0$; on $L^q_\sigma$, $q\ne2$, self-adjointness is no longer available to give a sharp contraction constant $1$ as in the Hilbert scale of Lemma~\ref{lem:MittagLeffler_diff}, so we retain the general constant $M$), and set
\[
K_0 := M\max\Bigl(\|u_0\|_{L^q_\sigma},\ \sup_{t\in[0,T]}\|Z_k(t)\|_{L^q_\sigma}\Bigr),
\]
  almost surely finite by Lemma~\ref{lem:Zk-supmom}. Define the Picard iterates $v^0:=u_0$ and, for $j\ge0$,
\[
v^{j+1}(t) = E_\al(t)u_0 + \int_0^t\cS_\al(t-s)\bigl[B(v^j(s)+Z_k(s))+f(v^j(s)+Z_k(s))\bigr]\,ds.
\]
Using $E_\al(t)=\int_0^\infty\xi_\al(\theta)T(t^\al\theta)\,d\theta$, then \eqref{eq:semigroup-Lp} at $\beta=0$, $r=q$, and finally the moment formula \eqref{eq:moment} at $\varrho=0$,
\[
\|E_\al(t)u_0\|_{L^q_\sigma}\le M\|u_0\|_{L^q_\sigma}\int_0^\infty\xi_\al(\theta)\,d\theta = M\|u_0\|_{L^q_\sigma}\le K_0 .
\] Combined with \eqref{eq:bilinear-time} and the analogous bound for $f$,
\begin{align*}
\|v^{j+1}(t)\|_{L^q_\sigma} \le K_0 &+ \int_0^t\frac{M}{(t-s)^{1-\al(\frac12-\frac1q)}}\bigl(\|v^j(s)\|_{L^q_\sigma}+\|Z_k(s)\|_{L^q_\sigma}\bigr)^2\,ds\\
&+ \int_0^t\frac{M}{(t-s)^{1-\al}}L_f\bigl(1+\|v^j(s)\|_{L^q_\sigma}+\|Z_k(s)\|_{L^q_\sigma}\bigr)\,ds.
\end{align*}
If $\sup_{[0,T^*]}\|v^j(t)\|_{L^q_\sigma}\le2K_0$, then, integrating the singular kernels explicitly (both exponents in $(0,1)$, using $\|v^j(s)\|_{L^q_\sigma}+\|Z_k(s)\|_{L^q_\sigma}\le2K_0+K_0\le3K_0$, absorbed into the constant),
\[
\sup_{t\in[0,T^*]}\|v^{j+1}(t)\|_{L^q_\sigma} \le K_0 + C_1M(2K_0)^2(T^*)^{\al(\frac12-\frac1q)} + C_2ML_f(1+2K_0)(T^*)^\al.
\]
Choosing $T^*$ small enough  (a random time, a measurable function of $K_0(\omega)$; it is \emph{not} a stopping time, since $K_0$ involves $\sup_{[0,T]}\|Z_k\|$) that the last two terms sum to at most $K_0$ gives $\sup_{[0,T^*]}\|v^{j+1}\|\le2K_0$, so the ball $B_{2K_0}\subset C([0,T^*];L^q_\sigma)$ is invariant under the Picard map, for every $j$ by induction (the base case $j=0$, $v^0=u_0$, satisfies $\|u_0\|_{L^q_\sigma}\le K_0\le2K_0$ trivially).

For the contraction estimate, write $\delta v^{j}:=v^{j+1}-v^j$ and, using bilinearity of $B$ (equivalently, of $\mathrm{div}(w_1\otimes w_2)$),
\begin{align*}
&A^{-1/2}P_H\,\mathrm{div}\bigl[(v^{j+1}+Z_k)\otimes(v^{j+1}+Z_k)-(v^j+Z_k)\otimes(v^j+Z_k)\bigr]\\
&\qquad = A^{-1/2}P_H\,\mathrm{div}\bigl[\delta v^j\otimes(v^{j+1}+Z_k)+(v^j+Z_k)\otimes\delta v^j\bigr].
\end{align*}
By the same Calder\'on-Zygmund boundedness as in the proof of Lemma~\ref{lem:bilinear} (Cauchy-Schwarz applied to $w_1\otimes w_2$ instead of $w\otimes w$), the $L^{q/2}_\sigma$-norm of the right-hand side is bounded by $C\|\delta v^j(s)\|_{L^q_\sigma}\bigl(\|v^{j+1}(s)+Z_k(s)\|_{L^q_\sigma}+\|v^j(s)+Z_k(s)\|_{L^q_\sigma}\bigr)\le C\cdot6K_0\cdot\|\delta v^j(s)\|_{L^q_\sigma}$, using both iterates lie in $B_{2K_0}$. Substituting into the Duhamel formula for $v^{j+2}-v^{j+1}$ -- with $A^{1/2}\cS_\al(t-s)$ applied exactly as in \eqref{eq:bilinear-time} -- and using the Lipschitz bound of Assumption~\ref{ass:f} for the analogous $f$-difference,
\[
\sup_{t\in[0,T^*]}\|v^{j+2}(t)-v^{j+1}(t)\|_{L^q_\sigma} \le \bigl(C_3M K_0(T^*)^{\al(\frac12-\frac1q)}+C_4ML_f(T^*)^\al\bigr)\sup_{t\in[0,T^*]}\|v^{j+1}(t)-v^j(t)\|_{L^q_\sigma},
\]
by the same explicit integration of the singular kernels as above. Since $T^*$ was already chosen depending only on $K_0(\omega)$, it may be shrunk further  (still a function of $K_0$ alone; not a stopping time) so that the bracketed factor is $<1$, uniformly in $j$. The Picard iterates then form a Cauchy sequence in $C([0,T^*];L^q_\sigma)$, a.s., converging to the unique fixed point $v$.
\end{proof}

Theorem~\ref{thm:exist_final}, combined with the Sobolev embedding used to bound $Z_k$ in $L^q_\sigma$, gives Theorem~\ref{thm:exist}.

\subsection{Continuation}

\begin{thm}[Continuation and maximal solution]\label{thm:continuation}
Under the assumptions of Theorem~\ref{thm:exist_final} there exist a maximal time $\tau^*\in(0,\infty]$ and a unique maximal mild solution $v\in C([0,\tau^*);L^q_\sigma)$ of \eqref{eq:v-mild}. On $\{\tau^*<\infty\}$, $\limsup_{t\uparrow\tau^*}\|v(t)\|_{L^q_\sigma}=\infty$. No condition on the parameters beyond $\al+2H>2$, $2<q<\al/(1-H)$ is required.
\end{thm}
\begin{proof}
The Caputo equation is non-local in time and $E_\al$ is not a semigroup, so the fixed point cannot be restarted at $T^*$ with initial datum $u(T^*)$; the notion of solution is the mild formulation, and the continuation must keep the memory of the past, as follows.

Let $T_1$ be a time up to which $v$ exists, $w:=v+Z_k$, $G(w):=B(w)+f(w)$. For $t\ge T_1$ the memory of the past is kept as a known forcing:
\[
v(t)=h_{T_1}(t)+\int_{T_1}^t\cS_\al(t-s)\,G(w(s))\,ds,\qquad
h_{T_1}(t):=E_\al(t)u_0+\int_0^{T_1}\cS_\al(t-s)\,G(w(s))\,ds .
\]
\emph{Step 1.} $h_{T_1}\in C([T_1,T_1+1];L^q_\sigma)$: $E_\al(\cdot)u_0$ is strongly continuous, and for $T_1\le t_1<t_2$ the near/far computation of Lemma~\ref{lem:holder-nonlinear} (integrating only over $s\le T_1\le t_1$, a subset of the original domain, of a nonnegative integrand) gives $\bigl\|\int_0^{T_1}[\cS_\al(t_2-s)-\cS_\al(t_1-s)]G(w(s))\,ds\bigr\|\le C(\sup_{s\le T_1}\|w(s)\|^2+1+\sup_{s\le T_1}\|w(s)\|)(t_2-t_1)^{\theta_B}$.

\emph{Step 2.} $\sup_{[T_1,T_1+1]}\|h_{T_1}\|\le\Lambda\bigl(\sup_{s\le T_1}\|w(s)\|\bigr)$ for a nondecreasing function $\Lambda$ (use \eqref{eq:bilinear-time} and its analogue for $f$).

\emph{Step 3.} Let $\rho:=\sup_{[T_1,T_1+1]}\|h_{T_1}\|+1+\sup_{[0,T_1+1]}\|Z_k\|$. On the ball of radius $1$ around $h_{T_1}$ in $C([T_1,T_1+\delta];L^q_\sigma)$ the map $v\mapsto h_{T_1}+\int_{T_1}^\cdot\cS_\al(\cdot-s)G(v+Z_k)\,ds$ is invariant if $C\rho^2\delta^{\theta_B}+CL_f(1+\rho)\delta^\al\le1$ and a strict contraction if $C\rho\,\delta^{\theta_B}+CL_f\delta^\al<1$, exactly as in the proof of Theorem~\ref{thm:exist_final}. Thus $\delta$ depends only on $\sup_{s\le T_1}\|w(s)\|$ and on $\sup_{[0,T_1+1]}\|Z_k\|$. The concatenation solves the mild equation on $[0,T_1+\delta]$.

\emph{Step 4.} Let $\tau^*$ be the supremum of existence times (uniqueness makes extensions consistent). If $\tau^*<\infty$ and $\sup_{[0,\tau^*)}\|v\|<\infty$, Step 3 gives a $\delta$ uniform for $T_1$ close to $\tau^*$, contradicting maximality.
\end{proof}

\section{H\"older regularity of the solution}\label{sec:holder}

Recall $v:=u-Z_k$ (Section~\ref{sec:existence}), so $u(t_2)-u(t_1)=[v(t_2)-v(t_1)]+[Z_k(t_2)-Z_k(t_1)]$.

\subsection{The noise term}

Transporting Theorem~\ref{thm:reg2p} from $\dot H^{\nu_q}$ ($\nu_q=1-\frac2q$) to $L^q_\sigma$ via Sobolev embedding,
\begin{equation}
\label{eq:noise-holder}
\|Z_k(t_2)-Z_k(t_1)\|_{L^p(\Omega;L^q_\sigma)} \le C_{p,k}\,(t_2-t_1)^{\gamma_q}, \qquad \gamma_q := \frac{\al}{q}+H-1.
\end{equation}

\subsection{The nonlinear and forcing terms}

By Theorem~\ref{thm:exist_final}'s proof, $v(s)$ and $Z_k(s)$ are bounded a.s.\ by $2K_0$ uniformly in $s\in[0,T^*]$, so $\|B(v(s)+Z_k(s))\|_{L^{q/2}_\sigma}\le4K_0^2$ and $\|f(v(s)+Z_k(s))\|_{L^q_\sigma}\le C(1+2K_0)$. These two terms use \emph{different} smoothing exponents for $\cS_\al$ -- $\beta=\frac12$ (mapping $L^{q/2}_\sigma\to L^q_\sigma$) for $B$, and $\beta=0$ (mapping $L^q_\sigma\to L^q_\sigma$) for $f$ -- and must be tracked separately. Write
\begin{align*}
v(t_2)-v(t_1) \;=\;& \underbrace{\bigl[E_\al(t_2)-E_\al(t_1)\bigr]u_0}_{J_1} \\[2pt]
&+ \underbrace{\int_0^{t_1}\bigl[\cS_\al(t_2-s)-\cS_\al(t_1-s)\bigr]B(w(s))\,ds}_{J_2}
\;+\; \underbrace{\int_{t_1}^{t_2}\cS_\al(t_2-s)B(w(s))\,ds}_{J_3} \\[2pt]
&+ \underbrace{\int_0^{t_1}\bigl[\cS_\al(t_2-s)-\cS_\al(t_1-s)\bigr]f(w(s))\,ds}_{J_4}
\;+\; \underbrace{\int_{t_1}^{t_2}\cS_\al(t_2-s)f(w(s))\,ds}_{J_5},
\end{align*}
where $w(s):=v(s)+Z_k(s)$, and set $\theta_B:=\al(\frac12-\frac1q)\in(0,\al)$, the smoothing exponent at $\beta=\frac12$ from Lemma~\ref{lem:Lq-smoothing}.

\begin{lemma}\label{lem:holder-nonlinear}
Almost surely,
\[
\|J_2+J_3\|_{L^q_\sigma}\le C\,K_0^2\,(t_2-t_1)^{\theta_B}, \qquad
\|J_4+J_5\|_{L^q_\sigma}\le C\,(1+K_0)\,(t_2-t_1)^{\al}.
\]
\end{lemma}

\begin{proof}
\emph{The terms $J_2,J_3$ (nonlinear part).} Both are Bochner integrals against a uniformly bounded function, so the near/far decomposition of Theorem~\ref{thm:reg2p}'s proof applies directly -- here in the simpler deterministic setting, where splitting the single integral $\int_0^{t_1}$ into two ranges is just linearity plus the triangle inequality, with no cross term to control. We use $\|d\cS_\al(r)/dr\|_{\cL(L^{q/2}_\sigma,L^q_\sigma)}\le Cr^{\theta_B-2}$, the single-operator, $\beta=\frac12$ analogue of Lemma~\ref{lem:HS_deriv}, obtained the same way. Writing $\tau:=t_2-t_1$:

\emph{Far} ($s\le t_1-\tau$): the mean value theorem gives
$\|\cS_\al(t_2-s)-\cS_\al(t_1-s)\|\le C\tau(t_1-s)^{\theta_B-2}$, whence
$\int_\tau^{t_1}\tau(t_1-s)^{\theta_B-2}\,ds\le C\tau^{\theta_B}$.

\emph{Near} ($t_1-\tau<s\le t_1$): using $t_2-s\ge\tau$,
$\|\cS_\al(t_2-s)\|+\|\cS_\al(t_1-s)\|\le C\tau^{\theta_B-1}+C(t_1-s)^{\theta_B-1}$, whence
$\int_{t_1-\tau}^{t_1}\bigl[\cdots\bigr]ds\le C\tau^{\theta_B}$.

Both regions contribute order $\tau^{\theta_B}$; multiplying by the uniform bound $\|B(w(s))\|_{L^{q/2}_\sigma}\le4K_0^2$ gives the first estimate.

\emph{The terms $J_4,J_5$ (forcing part).} The computation is identical, with $\beta=0$ in place of $\beta=\frac12$: Lemma~\ref{lem:Lq-smoothing} then gives the exponent $\al-1$ instead of $\theta_B-1$, and its derivative analogue $\|d\cS_\al(r)/dr\|_{\cL(L^{q}_\sigma,L^q_\sigma)}\le Cr^{\al-2}$. The same two regions each contribute order $\tau^{\al}$, and multiplying by $\|f(w(s))\|_{L^q_\sigma}\le C(1+2K_0)$ gives the second estimate.
\end{proof}

Since $\theta_B=\al(\frac12-\frac1q)<\al$, the nonlinear terms $J_2,J_3$ dominate: the combined contribution of $J_2,\dots,J_5$ is $O\bigl((t_2-t_1)^{\theta_B}\bigr)$.

\subsection{The initial term}

$E_\al(t)$ is not singular at $t=0$ (Lemma~\ref{lem:MittagLeffler_diff}), so $J_1$ presents no endpoint difficulty; but $u_0\in L^q_\sigma$ alone gives no positive H\"older rate for $E_\al(\cdot)u_0$,  since $E_\al$ is merely bounded by $M$, not a smoothing operator, on $L^q_\sigma$. We assume in addition $u_0\in D(A_q^\theta)$ (domain in $L^q_\sigma$) for some $\theta\in(0,1]$; by Lemma~\ref{lem:MittagLeffler_diff} (in $L^q_\sigma$),
\[
\|[E_\al(t_2)-E_\al(t_1)]u_0\|_{L^q_\sigma} \le C(t_2-t_1)^{\al\theta}\|u_0\|_{D(A^\theta)}.
\]

\subsection{Main result}

\begin{thm}[H\"older regularity]\label{thm:holder_final}
Under the assumptions of Theorem~\ref{thm:exist}, suppose in addition $u_0\in D(A_q^\theta)$, $\theta\in(0,1]$. Then for every $T>0$ and $0\le t_1<t_2\le T$,
\[
\bigl(\E\|u(t_2\wedge T^*)-u(t_1\wedge T^*)\|_{L^q_\sigma}^p\bigr)^{1/p}\le C\,(t_2-t_1)^\beta,\qquad
\beta=\min\Bigl\{\al\theta,\ \al\bigl(\tfrac12-\tfrac1q\bigr),\ \tfrac\al q+H-1\Bigr\}.
\]
\end{thm}
\begin{proof}
If $t_1<T^*<t_2$, the difference is $u(T^*)-u(t_1)$ with $T^*-t_1\le t_2-t_1$, so it suffices to treat $t_1<t_2\le T^*$. The three contributions are combined by the triangle inequality in $L^p(\Omega;L^q_\sigma)$.

\emph{Initial term.} By the estimate just established, $\|J_1\|_{L^q_\sigma}\le C(t_2-t_1)^{\al\theta}\|u_0\|_{D(A^\theta)}$ pointwise in $\omega$ (indeed $u_0$ is deterministic), hence the same bound holds in $L^p(\Omega;L^q_\sigma)$.

\emph{Nonlinear and forcing terms.} Lemma~\ref{lem:holder-nonlinear} gives the almost-sure bounds $\|J_2+J_3\|_{L^q_\sigma}\le CK_0^2(t_2-t_1)^{\theta_B}$ and $\|J_4+J_5\|_{L^q_\sigma}\le C(1+K_0)(t_2-t_1)^{\al}$. Raising to the power $p$ and taking expectations,
\[
\bigl\|J_2+J_3\bigr\|_{L^p(\Omega;L^q_\sigma)} \le C\,\bigl(\E[K_0^{2p}]\bigr)^{1/p}(t_2-t_1)^{\theta_B},
\]
and $\E[K_0^{2p}]<\infty$ by Lemma~\ref{lem:Zk-supmom}. The terms $J_4,J_5$ are handled identically, with $\E[(1+K_0)^p]<\infty$ and the larger exponent $\al$.

\emph{Noise term.} This is \eqref{eq:noise-holder}, already an $L^p(\Omega;L^q_\sigma)$ estimate with exponent $\gamma_q=\frac\al q+H-1$.

\emph{Conclusion.} Summing the three and using $(t_2-t_1)^{a}\le T^{a-\beta}(t_2-t_1)^{\beta}$ for each exponent $a\in\{\al\theta,\theta_B,\al,\gamma_q\}$ (all $\ge\beta$, with $T\le T^*$ bounded), every contribution is dominated by a constant multiple of $(t_2-t_1)^{\beta}$ with $\beta=\min\{\al\theta,\theta_B,\gamma_q\}$; the exponent $\al$ carried by $J_4,J_5$ never enters the minimum, since $\theta_B=\al(\frac12-\frac1q)<\al$.
\end{proof}

This proves Theorem~\ref{thm:holder}.

\section{Higher-order moment estimates}\label{sec:moments}

We prove Theorem~\ref{thm:moment}. Throughout, $T>0$ is fixed and $K_0$ is the random constant of Section~\ref{sec:existence}, which satisfies $K_0\le M\bigl(\|u_0\|_{L^q_\sigma}+\sup_{t\le T}\|Z_k(t)\|_{L^q_\sigma}\bigr)$.

\emph{(a)} is the quantitative part of Lemma~\ref{lem:Zk-supmom}.

\emph{(b)} By Lemma~\ref{lem:bilinear}, $\|B(Z_k(t))\|_{L^{q/2}_\sigma}\le C\|Z_k(t)\|_{L^q_\sigma}^2$, hence $\E\|B(Z_k(t))\|^p\le C^p\E\|Z_k(t)\|_{L^q_\sigma}^{2p}\le C^p(2p-1)^{kp}C_0^{2p}t^{2p\gamma_q}$, by Theorem~\ref{thm:reg1p} at order $2p$ with $\nu=\nu_q$ and the embedding $\dot H^{\nu_q}\hookrightarrow L^q_\sigma$.

\emph{(c)} On $[0,T^*]$: $\|u\|\le\|v\|+\|Z_k\|\le3K_0$ with $K_0\le M(\|u_0\|+\sup_{t\le T}\|Z_k(t)\|)$. Take $L^p(\Omega)$ norms and use (a).

\textbf{Remark.} By the product formula for multiple Wiener--It\^o integrals (Nourdin--Peccati), $B(Z_k)$ lies in the sum of chaoses of orders $2k,2k-2,\dots,0$; hypercontractivity bounds $\|B(Z_k)\|_{L^p}$ by $C(p-1)^k$, which is the origin of the exponent $k$ in (b). No matching lower bound is proved. With $N_p(t):=\E\|B(Z_k(t))\|_{L^{q/2}_\sigma}^p$, part (b) gives $\frac1p\log N_p(t)\le k\log p+C$ as $p\to\infty$.

\section{Non-central limit theorem}\label{sec:nclt}

In this section we prove that the Hermite-driven fractional
Navier--Stokes equation arises as a scaling limit of equations driven
by nonlinear functionals of stationary Gaussian sequences with
long-range dependence.

To keep the principal conclusions visible, we first state the three
main results of the section.  Their proofs are given afterward.  The
first theorem gives functional convergence of the complete
cylindrical partial-sum process.  The second transfers this convergence
through the singular Mittag--Leffler convolution.  The third transfers
the resulting convergence through a radially localized nonlinear
solution map on a common deterministic interval.

For every Stokes mode \(j\geq1\), let
\(\{\xi_n^{(j)}\}_{n\in\mathbb Z}\) be independent copies of a
stationary, centered, unit-variance Gaussian sequence whose covariance
satisfies
\begin{equation}\label{eq:nclt-covariance}
 \rho(m)
 :=
 \E[\xi_0^{(j)}\xi_m^{(j)}]
 \sim
 m^{2H_0-2}L(m),
 \qquad m\to\infty,
\end{equation}
where \(L\) is positive and slowly varying and
\[
 H_0\in\left(1-\frac1{2k},1\right).
\]
Set
\begin{equation}\label{eq:H-H0}
 H:=1-k(1-H_0)\in\left(\frac12,1\right).
\end{equation}
Then
\[
 \rho(m)
 \sim
 m^{-2(1-H)/k}L(m).
\]

Let \(H_k\) be the Hermite polynomial of order \(k\), normalized by
\[
 \E[H_k(\xi)^2]=k!,
 \qquad
 \xi\sim N(0,1),
\]
and define
\begin{equation}\label{eq:dN-definition}
 d_N^2
 :=
 \Var\left(
 \sum_{n=1}^N H_k(\xi_n^{(1)})
 \right).
\end{equation}
For \(t\geq0\), with \(m=\lfloor Nt\rfloor\), define
\begin{equation}\label{eq:SN-interpolation}
 \widetilde S_N^{(j)}(t)
 :=
 \frac1{d_N}
 \left[
 \sum_{n=1}^{m}H_k(\xi_n^{(j)})
 +(Nt-m)H_k(\xi_{m+1}^{(j)})
 \right],
\end{equation}
and set formally
\[
 \widetilde S_N(t)
 :=
 \sum_{j\geq1}\widetilde S_N^{(j)}(t)e_j.
\]

For a Banach space \(E\), let
\(c_0^\beta([0,T];E)\) denote the separable little H\"older space of
\(E\)-valued paths vanishing at zero.

\begin{thm}[Functional cylindrical Hermite non-central limit theorem]
\label{thm:functional-hermite-nclt}
Assume \eqref{eq:nclt-covariance}. Let \(T>0\), \(0<\beta<H\), and \(\vartheta>1\).  Then
\begin{equation}\label{eq:scalar-functional-nclt}
 \widetilde S_N^{(j)}
 \Longrightarrow
 Z_{H,j}^{(k)}
 \quad\text{in }
 c_0^\beta([0,T];\mathbb R)
\end{equation}
for every \(j\geq1\), and
\begin{equation}\label{eq:cyl-functional-nclt}
 \widetilde S_N
 \Longrightarrow
 Z_H^k
 :=
 \sum_{j\geq1}Z_{H,j}^{(k)}e_j
 \quad\text{in }
 c_0^\beta([0,T];\dot H^{-\vartheta}).
\end{equation}
Moreover,
\begin{equation}\label{eq:dN-asymptotic-main}
 d_N^2
 \sim
 \frac{2k!}{(2H-1)(2H)}
 N^{2H}L(N)^k.
\end{equation}
\end{thm}

Set
\[
 \nu_q:=1-\frac2q.
\]
The singularity of the Mittag--Leffler kernel prevents a direct continuous-mapping argument for the stochastic convolution. We therefore introduce an endpoint-compensated Volterra map. For \(z\in c_0^\beta([0,T];\dot H^{-\vartheta})\), define the
endpoint-compensated Volterra map by
\begin{equation}\label{eq:Valpha-definition}
 (\mathcal V_\alpha z)(t)
 :=
 \mathcal S_\alpha(t)z(t)
 +
 \int_0^t
 \mathcal S_\alpha'(t-r)
 [z(r)-z(t)]\,dr.
\end{equation}
Define the approximating stochastic convolution by
\begin{equation}\label{eq:ZkN-definition}
 Z_k^N:=\mathcal V_\alpha(\widetilde S_N).
\end{equation}

\begin{thm}[Approximation of the stochastic convolution]
\label{thm:convolution-nclt}
Assume
\begin{equation}\label{eq:q-range}
 0<\alpha<1,
 \qquad
 H\in\left(\frac12,1\right),
 \qquad
 2<q<\frac{\alpha}{1-H}.
\end{equation}
Then, for every finite \(T>0\),
\begin{equation}\label{eq:ZkN-convergence}
 Z_k^N
 \Longrightarrow
 Z_k
 \quad\text{in }
 C([0,T];L_\sigma^q).
\end{equation}
\end{thm}

For \(\tau>0\), set
\[
 X_\tau:=C([0,\tau];L_\sigma^q),
 \qquad
 \|z\|_{X_\tau}
 :=
 \sup_{0\leq t\leq\tau}\|z(t)\|_{L_\sigma^q}.
\]
For \(R>0\), define the radial projection
\begin{equation}\label{eq:PiR}
 \Pi_R^\tau z
 :=
 \begin{cases}
 z,
 &\|z\|_{X_\tau}\leq R,
 \\[1mm]
 \displaystyle
 \frac{Rz}{\|z\|_{X_\tau}},
 &\|z\|_{X_\tau}>R.
 \end{cases}
\end{equation}

Given \(z\in X_\tau\), consider the localized mild equation
\begin{equation}\label{eq:localized-mild-equation}
 \begin{split}
 u_R^z(t)
 &=
 E_\alpha(t)u_0
 +
 \int_0^t
 \mathcal S_\alpha(t-s)
 \bigl[
 B(u_R^z(s))+f(u_R^z(s))
 \bigr]\,ds
 \\
 &\qquad
 +
 (\Pi_R^\tau z)(t).
 \end{split}
\end{equation}

\begin{thm}[Radially localized solution approximation]
\label{thm:localized-solution-nclt}
Assume \eqref{eq:q-range}, let \(u_0\in L_\sigma^q\), and suppose
that \(f:L_\sigma^q\to L_\sigma^q\) satisfies Assumption~5.1.

For every \(R>0\), there exist deterministic constants
\(T_R\in(0,1]\) and \(M_R>0\) such that:

\begin{enumerate}

\item
for every \(z\in X_{T_R}\), equation
\eqref{eq:localized-mild-equation}, with \(\tau=T_R\), has a unique
fixed point
\[
 u_R^z
 \in
 \mathbb B_{M_R}
 :=
 \left\{
 u\in X_{T_R}:
 \|u\|_{X_{T_R}}\leq M_R
 \right\};
\]

\item
the localized solution map
\[
 \mathfrak S_R:X_{T_R}\longrightarrow X_{T_R},
 \qquad
 \mathfrak S_R(z):=u_R^z,
\]
is globally Lipschitz;

\item
if
\[
 u_R^N
 :=
 \mathfrak S_R
 \left(
 Z_k^N|_{[0,T_R]}
 \right),
 \qquad
 u_R
 :=
 \mathfrak S_R
 \left(
 Z_k|_{[0,T_R]}
 \right),
\]
then
\begin{equation}\label{eq:localized-solution-convergence}
 u_R^N
 \Longrightarrow
 u_R
 \quad\text{in }
 C([0,T_R];L_\sigma^q).
\end{equation}

\end{enumerate}
\end{thm}

% ============================================================
\subsection{Proof of the functional non-central limit theorem}
% ============================================================

We now prove Theorem~\ref{thm:functional-hermite-nclt}.

For jointly standard Gaussian random variables \(X,Y\) satisfying
\(\E[XY]=r\), the Hermite covariance identity gives
\begin{equation}\label{eq:Hermite-covariance-identity}
 \E[H_k(X)H_k(Y)]
 =
 k!r^k.
\end{equation}
Consequently,
\begin{equation}\label{eq:Hermite-sequence-covariance}
 \E[
 H_k(\xi_0^{(j)})
 H_k(\xi_m^{(j)})
 ]
 =
 k!\rho(m)^k.
\end{equation}

\subsubsection*{Step 1: asymptotics of the normalization}

By stationarity and
\eqref{eq:Hermite-sequence-covariance},
\begin{equation}\label{eq:dN-expansion}
 d_N^2
 =
 k!\left[
 N+
 2\sum_{m=1}^{N-1}(N-m)\rho(m)^k
 \right].
\end{equation}
By \eqref{eq:H-H0},
\[
 \rho(m)^k
 \sim
 m^{2H-2}L(m)^k.
\]
Since \(2H-2\in(-1,0)\), Karamata's theorem gives
\[
 \sum_{m=1}^{N-1}m^{2H-2}L(m)^k
 \sim
 \frac{N^{2H-1}L(N)^k}{2H-1}
\]
and
\[
 \sum_{m=1}^{N-1}m^{2H-1}L(m)^k
 \sim
 \frac{N^{2H}L(N)^k}{2H}.
\]
Therefore,
\begin{align}
 \sum_{m=1}^{N-1}(N-m)\rho(m)^k
 &\sim
 N^{2H}L(N)^k
 \left(
 \frac1{2H-1}-\frac1{2H}
 \right)
 \nonumber\\
 &=
 \frac{N^{2H}L(N)^k}{(2H-1)(2H)}.
 \label{eq:weighted-Karamata}
\end{align}
Since \(2H>1\), the term \(N\) in
\eqref{eq:dN-expansion} is of lower order.  Hence
\[
 d_N^2
 \sim
 \frac{2k!}{(2H-1)(2H)}
 N^{2H}L(N)^k.
\]
In particular,
\begin{equation}\label{eq:dN-RV}
 d_N
 \sim
 \left(
 \frac{2k!}{(2H-1)(2H)}
 \right)^{1/2}
 N^HL(N)^{k/2},
\end{equation}
and \((d_N)\) is regularly varying with index \(H\).

\subsubsection*{Step 2: finite-dimensional convergence}

The Hermite polynomial \(H_k\) has Hermite rank \(k\), and
\[
 k(2-2H_0)
 =
 2(1-H)
 \in(0,1).
\]
The Dobrushin--Major--Taqqu theorem, with normalization \(d_N\), gives
\[
 \bigl(
 \widetilde S_N^{(j)}(t_1),\ldots,
 \widetilde S_N^{(j)}(t_m)
 \bigr)
 \Longrightarrow
 \bigl(
 Z_{H,j}^{(k)}(t_1),\ldots,
 Z_{H,j}^{(k)}(t_m)
 \bigr)
\]
for every finite collection \(t_1,\ldots,t_m\in[0,T]\).

The interpolation term does not affect the finite-dimensional limit,
because its \(L^2(\Omega)\)-norm is bounded by
\[
 \frac{\sqrt{k!}}{d_N}
 \longrightarrow0.
\]

\subsubsection*{Step 3: scalar tightness}

At mesh points, stationarity gives, for integers
\(0\leq r<s\leq NT\),
\begin{equation}\label{eq:mesh-increment}
 \left\|
 \widetilde S_N^{(j)}
 \left(\frac{s}{N}\right)
 -
 \widetilde S_N^{(j)}
 \left(\frac{r}{N}\right)
 \right\|_{L^2(\Omega)}
 =
 \frac{d_{s-r}}{d_N}.
\end{equation}
Since \(d_N\) is regularly varying with index \(H\), Potter's bound
implies that, for every \(\varepsilon>0\),
\begin{equation}\label{eq:Potter-dN}
 \frac{d_m}{d_N}
 \leq
 C_{\varepsilon,T}
 \left(\frac mN\right)^{H-\varepsilon},
 \qquad
 1\leq m\leq NT.
\end{equation}

If \(s,t\) belong to the same mesh interval, then
\[
 \|
 \widetilde S_N^{(j)}(t)
 -
 \widetilde S_N^{(j)}(s)
 \|_{L^2}
 =
 \frac{N|t-s|}{d_N}\sqrt{k!}
 \leq
 C_{\varepsilon,T}|t-s|^{H-\varepsilon}.
\]
For arbitrary \(s<t\), splitting the increment into two fractional
mesh increments and one mesh-point increment gives
\begin{equation}\label{eq:scalar-L2-increment}
 \sup_{N,j}
 \|
 \widetilde S_N^{(j)}(t)
 -
 \widetilde S_N^{(j)}(s)
 \|_{L^2}
 \leq
 C_{\varepsilon,T}|t-s|^{H-\varepsilon}.
\end{equation}

Every increment belongs to the \(k\)-th homogeneous Wiener chaos.
Lemma~2.8 therefore gives, for every \(p\geq2\),
\begin{equation}\label{eq:scalar-Lp-increment}
 \sup_{N,j}
 \E\left|
 \widetilde S_N^{(j)}(t)
 -
 \widetilde S_N^{(j)}(s)
 \right|^p
 \leq
 C_{p,\varepsilon,T}
 |t-s|^{p(H-\varepsilon)}.
\end{equation}

Fix \(0<\beta<H\).  Choose \(\varepsilon>0\), \(p\geq2\), and
\(\gamma\) such that
\[
 \beta<\gamma<H-\varepsilon-\frac1p.
\]
The quantitative Kolmogorov--Chentsov theorem gives
\begin{equation}\label{eq:uniform-Cgamma}
 \sup_{N,j}
 \E
 \|\widetilde S_N^{(j)}\|_{C^\gamma([0,T])}^p
 <\infty.
\end{equation}
Since
\[
 C^\gamma([0,T])
 \Subset
 c_0^\beta([0,T]),
 \qquad
 \gamma>\beta,
\]
the scalar laws are tight in \(c_0^\beta([0,T])\).  Together with
finite-dimensional convergence, this proves
\eqref{eq:scalar-functional-nclt}.

\subsubsection*{Step 4: cylindrical convergence}

By the two-dimensional Weyl asymptotics recorded in Section~2.1,
\[
 \lambda_j\asymp j.
\]
Hence
\begin{equation}\label{eq:spatial-summability}
 \sum_{j\geq1}\lambda_j^{-\vartheta}<\infty,
 \qquad
 \vartheta>1.
\end{equation}

Choose
\[
 \beta<\gamma<H.
\]
By \eqref{eq:uniform-Cgamma}, with a sufficiently large moment order,
\begin{equation}\label{eq:uniform-mode-Cgamma}
 \sup_{N,j}
 \E
 \|\widetilde S_N^{(j)}\|_{C^\gamma([0,T])}^2
 <\infty.
\end{equation}
For \(M>J\),
\begin{align}
 &\left\|
 \sum_{j=J+1}^{M}
 \widetilde S_N^{(j)}e_j
 \right\|_
 {C^\beta([0,T];\dot H^{-\vartheta})}^2
 \nonumber\\
 &\qquad\leq
 C_{\beta,\gamma,T}
 \sum_{j=J+1}^{M}
 \lambda_j^{-\vartheta}
 \|
 \widetilde S_N^{(j)}
 \|_{C^\gamma([0,T])}^2.
 \label{eq:cyl-tail-bound}
\end{align}
Taking expectations and using
\eqref{eq:uniform-mode-Cgamma} gives
\begin{equation}\label{eq:prelimit-tail}
 \lim_{J\to\infty}
 \sup_N
 \E
 \left\|
 \sum_{j>J}
 \widetilde S_N^{(j)}e_j
 \right\|_
 {C^\beta([0,T];\dot H^{-\vartheta})}^2
 =
 0.
\end{equation}
The same argument, using Proposition~2.7(v), gives
\begin{equation}\label{eq:limit-tail}
 \lim_{J\to\infty}
 \E
 \left\|
 \sum_{j>J}
 Z_{H,j}^{(k)}e_j
 \right\|_
 {C^\beta([0,T];\dot H^{-\vartheta})}^2
 =
 0.
\end{equation}

For every fixed \(J\), scalar functional convergence and independence
of the modes imply joint convergence of the first \(J\) coordinates.
The converging-together theorem in the separable Banach space
\(c_0^\beta([0,T];\dot H^{-\vartheta})\), combined with
\eqref{eq:prelimit-tail}--\eqref{eq:limit-tail}, proves
\eqref{eq:cyl-functional-nclt}.
\qed

% ============================================================
\subsection{Proof of the stochastic-convolution approximation}
% ============================================================

We now prove Theorem~\ref{thm:convolution-nclt}.

Choose \(0<\beta<H\) and \(\vartheta>1\), sufficiently close to
\(H\) and \(1\), respectively, so that
\begin{equation}\label{eq:parameter-condition}
 \nu_q+\vartheta
 <
 \frac{2(\alpha+\beta-1)}{\alpha}.
\end{equation}
Such a choice is possible because, in the limit
\(\beta\uparrow H\) and \(\vartheta\downarrow1\),
condition \eqref{eq:parameter-condition} becomes
\[
 2-\frac2q
 <
 \frac{2(\alpha+H-1)}{\alpha},
\]
which is equivalent to
\[
 q<\frac{\alpha}{1-H}.
\]

Set
\begin{equation}\label{eq:mq-eta}
 m_q:=\nu_q+\vartheta,
 \qquad
 \eta:=
 \alpha+\beta-1-\frac{\alpha m_q}{2}.
\end{equation}
Then
\[
 0<m_q<2,
 \qquad
 \eta>0.
\]

\begin{lemma}[Continuity of the compensated Volterra map]
\label{lem:Valpha-bounded}
Under \eqref{eq:parameter-condition},
\[
 \mathcal V_\alpha:
 c_0^\beta([0,T];\dot H^{-\vartheta})
 \longrightarrow
 C([0,T];\dot H^{\nu_q})
\]
is bounded and linear.  More precisely,
\begin{equation}\label{eq:Valpha-bound}
 \|\mathcal V_\alpha z\|_
 {C([0,T];\dot H^{\nu_q})}
 \leq
 C_T
 \|z\|_
 {C^\beta([0,T];\dot H^{-\vartheta})}.
\end{equation}
\end{lemma}

\begin{proof}
The spectral representation is
\[
 \mathcal S_\alpha(r)e_j
 =
 r^{\alpha-1}
 E_{\alpha,\alpha}(-\lambda_jr^\alpha)e_j.
\]
Using
\[
 |E_{\alpha,\alpha}(-x)|
 \leq
 \frac{C}{1+x},
 \qquad
 x\geq0,
\]
we obtain
\begin{equation}\label{eq:Salpha-operator-bound}
 \|\mathcal S_\alpha(r)\|_
 {\mathcal L(\dot H^{-\vartheta},\dot H^{\nu_q})}
 \leq
 Cr^{\alpha-1-\alpha m_q/2}.
\end{equation}

Indeed,
\[
 \lambda_j^{m_q/2}
 r^{\alpha-1}
 |E_{\alpha,\alpha}(-\lambda_jr^\alpha)|
 \leq
 Cr^{\alpha-1-\alpha m_q/2}
 \frac{
 (\lambda_jr^\alpha)^{m_q/2}
 }{
 1+\lambda_jr^\alpha
 },
\]
and the final factor is uniformly bounded because \(0<m_q<2\).

The derivative identity
\[
 \frac{d}{dr}
 \left[
 r^{\alpha-1}
 E_{\alpha,\alpha}(-\lambda r^\alpha)
 \right]
 =
 r^{\alpha-2}
 E_{\alpha,\alpha-1}(-\lambda r^\alpha)
\]
and the corresponding negative-axis estimate give
\begin{equation}\label{eq:Salpha-prime-operator-bound}
 \|\mathcal S_\alpha'(r)\|_
 {\mathcal L(\dot H^{-\vartheta},\dot H^{\nu_q})}
 \leq
 Cr^{\alpha-2-\alpha m_q/2}.
\end{equation}

Since \(z(0)=0\),
\[
 \|z(t)\|_{\dot H^{-\vartheta}}
 \leq
 t^\beta
 \|z\|_{C^\beta([0,T];\dot H^{-\vartheta})}.
\]
Therefore,
\begin{equation}\label{eq:Valpha-first-term}
 \|\mathcal S_\alpha(t)z(t)\|_{\dot H^{\nu_q}}
 \leq
 Ct^\eta
 \|z\|_{C^\beta([0,T];\dot H^{-\vartheta})}.
\end{equation}
Moreover,
\begin{align*}
 &
 \left\|
 \int_0^t
 \mathcal S_\alpha'(t-r)
 [z(r)-z(t)]\,dr
 \right\|_{\dot H^{\nu_q}}
 \\
 &\qquad\leq
 C
 \|z\|_{C^\beta([0,T];\dot H^{-\vartheta})}
 \int_0^t
 (t-r)^{\eta-1}\,dr
 \\
 &\qquad=
 \frac{C}{\eta}
 t^\eta
 \|z\|_{C^\beta([0,T];\dot H^{-\vartheta})}.
\end{align*}
This proves \eqref{eq:Valpha-bound}.

Continuity is immediate for smooth finite-mode paths.  Such paths are
dense in the little H\"older space, and
\eqref{eq:Valpha-bound} permits passage to the limit.
\end{proof}

\begin{lemma}[Identification with the stochastic convolution]
\label{lem:Valpha-identification}
Under \eqref{eq:parameter-condition},
\begin{equation}\label{eq:Valpha-identification}
 \mathcal V_\alpha(Z_H^k)
 =
 Z_k
 \quad\text{in }
 C([0,T];\dot H^{\nu_q})
 \quad\text{a.s.}
\end{equation}
\end{lemma}

\begin{proof}
Fix \(j\geq1\), \(t>0\), and \(0<\varepsilon<t\).  Integration by
parts on \([0,t-\varepsilon]\) gives
\begin{align}
 &
 \int_0^{t-\varepsilon}
 s_{\alpha,j}(t-r)\,
 dZ_{H,j}^{(k)}(r)
 \nonumber\\
 &\quad=
 s_{\alpha,j}(t)Z_{H,j}^{(k)}(t)
 +
 \int_0^{t-\varepsilon}
 s_{\alpha,j}'(t-r)
 \bigl[
 Z_{H,j}^{(k)}(r)-Z_{H,j}^{(k)}(t)
 \bigr]\,dr
 \nonumber\\
 &\qquad+
 s_{\alpha,j}(\varepsilon)
 \bigl[
 Z_{H,j}^{(k)}(t-\varepsilon)
 -
 Z_{H,j}^{(k)}(t)
 \bigr].
 \label{eq:truncated-IBP}
\end{align}

Let \(P_J\) denote projection onto the first \(J\) Stokes modes.
By \eqref{eq:Salpha-operator-bound} and the
\(\beta\)-H\"older regularity of \(Z_H^k\),
\begin{align*}
 &
 \left\|
 \mathcal S_\alpha(\varepsilon)
 P_J
 \bigl[
 Z_H^k(t-\varepsilon)-Z_H^k(t)
 \bigr]
 \right\|_{\dot H^{\nu_q}}
 \\
 &\qquad\leq
 C\varepsilon^\eta
 \|P_JZ_H^k\|_
 {C^\beta([0,T];\dot H^{-\vartheta})}.
\end{align*}

The right-hand side converges to zero in \(L^2(\Omega)\), uniformly in \(J\), because the family \(\{P_J Z_H^k\}_{J\ge1}\) is uniformly bounded in \(L^2(\Omega;C^\beta([0,T];\dot H^{-\vartheta}))\), and the tail estimate of Theorem~\ref{thm:functional-hermite-nclt} implies convergence of the projected processes.

For fixed \(J\) and \(j\le J\), define
\[
g_{\varepsilon,j}(r)
:=
s_{\alpha,j}(t-r)\,
\mathbf 1_{[0,t-\varepsilon]}(r).
\]

By Lemma~4.4, the family
\(\{g_{\varepsilon,j}\}_{\varepsilon>0}\)
is Cauchy in the covariance Hilbert space of the scalar Hermite Wiener
integral, and

\[
g_{\varepsilon,j}
\longrightarrow
s_{\alpha,j}(t-\cdot)
\]

in that space as \(\varepsilon\downarrow0\).
Therefore Lemma~3.1 yields

\[
\int_0^{t-\varepsilon}
s_{\alpha,j}(t-r)\,
dZ^{(k)}_{H,j}(r)
\longrightarrow
\int_0^{t}
s_{\alpha,j}(t-r)\,
dZ^{(k)}_{H,j}(r)
\]

in \(L^2(\Omega)\).
Since only finitely many modes are involved, Lemma~3.4 allows the
corresponding finite-mode stochastic integrals to be viewed as
\(\dot H^{\nu_q}\)-valued random variables.

Passing first to the limit \(\varepsilon\downarrow0\) in
\eqref{eq:truncated-IBP}, and then to the limit \(J\to\infty\), gives
\[
 \mathcal V_\alpha(Z_H^k)(t)
 =
 Z_k(t)
 \quad\text{in }
 L^2(\Omega;\dot H^{\nu_q}).
\]
This equality holds simultaneously at all rational times.  Both
processes possess continuous \(\dot H^{\nu_q}\)-valued versions, so
they are indistinguishable.
\end{proof}

We can now prove Theorem~\ref{thm:convolution-nclt}.  By
Theorem~\ref{thm:functional-hermite-nclt},
\[
 \widetilde S_N
 \Longrightarrow
 Z_H^k
 \quad\text{in }
 c_0^\beta([0,T];\dot H^{-\vartheta}).
\]
By Lemma~\ref{lem:Valpha-bounded} and the continuous mapping theorem,
\[
 Z_k^N
 =
 \mathcal V_\alpha(\widetilde S_N)
 \Longrightarrow
 \mathcal V_\alpha(Z_H^k)
 =
 Z_k
\]
in \(C([0,T];\dot H^{\nu_q})\).  Finally,
Proposition~2.1 gives the continuous embedding
\[
 \dot H^{\nu_q}
 =
 \dot H^{1-2/q}
 \hookrightarrow
 L_\sigma^q.
\]
This proves \eqref{eq:ZkN-convergence}.
\qed

% ============================================================
\subsection{Proof of the radially localized solution approximation}
% ============================================================

We now prove Theorem~\ref{thm:localized-solution-nclt}.

The radial projection satisfies
\begin{equation}\label{eq:PiR-bounded}
 \|\Pi_R^\tau z\|_{X_\tau}\leq R
\end{equation}
and
\begin{equation}\label{eq:PiR-Lipschitz}
 \|\Pi_R^\tau z-\Pi_R^\tau\widetilde z\|_{X_\tau}
 \leq
 2\|z-\widetilde z\|_{X_\tau}.
\end{equation}

Set
\begin{equation}\label{eq:thetaB}
 \theta_B
 :=
 \alpha\left(\frac12-\frac1q\right)>0.
\end{equation}
By Lemma~5.3, Lemma~2.2, and equation~(16), there is \(C_B>0\)
such that
\begin{equation}\label{eq:B-time-estimate}
 \left\|
 \int_0^\cdot
 \mathcal S_\alpha(\cdot-s)B(u(s),v(s))\,ds
 \right\|_{X_\tau}
 \leq
 C_B\tau^{\theta_B}
 \|u\|_{X_\tau}\|v\|_{X_\tau}.
\end{equation}
By Lemma~2.2 and Assumption~5.1, there is \(C_f>0\) such that
\begin{equation}\label{eq:f-time-estimate}
 \left\|
 \int_0^\cdot
 \mathcal S_\alpha(\cdot-s)
 [f(u(s))-f(v(s))]\,ds
 \right\|_{X_\tau}
 \leq
 C_f\tau^\alpha
 \|u-v\|_{X_\tau}.
\end{equation}

Set
\[
 C_0
 :=
 \sup_{0\leq t\leq1}
 \|E_\alpha(t)\|_{\mathcal L(L_\sigma^q)}
\]
and
\begin{equation}\label{eq:MR}
 M_R
 :=
 2\left(
 C_0\|u_0\|_{L_\sigma^q}+R+1
 \right).
\end{equation}
For \(z\in X_{T_R}\), define
\[
 \begin{split}
 (\Gamma_z u)(t)
 &:=
 E_\alpha(t)u_0
 +
 \int_0^t
 \mathcal S_\alpha(t-s)
 [B(u(s))+f(u(s))]\,ds
 \\
 &\qquad+
 (\Pi_R^{T_R}z)(t).
 \end{split}
\]
For \(u\in\mathbb B_{M_R}\), equations
\eqref{eq:B-time-estimate}--\eqref{eq:f-time-estimate} and
Assumption~5.1 give
\begin{equation}\label{eq:Gamma-ball}
 \|\Gamma_z u\|_{X_{T_R}}
 \leq
 C_0\|u_0\|_q
 +
 R
 +
 C_BT_R^{\theta_B}M_R^2
 +
 C_fT_R^\alpha(1+M_R).
\end{equation}
For \(u,v\in\mathbb B_{M_R}\), the identity
\[
 B(u,u)-B(v,v)
 =
 B(u-v,u)+B(v,u-v)
\]
gives
\begin{equation}\label{eq:Gamma-contraction}
 \|\Gamma_z u-\Gamma_zv\|_{X_{T_R}}
 \leq
 \left(
 2C_BM_RT_R^{\theta_B}
 +
 C_fT_R^\alpha
 \right)
 \|u-v\|_{X_{T_R}}.
\end{equation}

Choose \(T_R\in(0,1]\) so that
\begin{align}
 C_BT_R^{\theta_B}M_R^2
 +
 C_fT_R^\alpha(1+M_R)
 &\leq1,
 \label{eq:TR-ball-condition}\\
 2C_BM_RT_R^{\theta_B}
 +
 C_fT_R^\alpha
 &\leq\frac12.
 \label{eq:TR-contraction-condition}
\end{align}
Then \(\Gamma_z\) maps \(\mathbb B_{M_R}\) into itself and is a
strict contraction there, uniformly in \(z\).  Banach's fixed-point
theorem gives a unique fixed point
\[
 u_R^z\in\mathbb B_{M_R}.
\]

Let \(z,\widetilde z\in X_{T_R}\).  Subtracting the two fixed-point
equations and using \eqref{eq:PiR-Lipschitz} gives
\begin{align*}
 \|u_R^z-u_R^{\widetilde z}\|_{X_{T_R}}
 &\leq
 2\|z-\widetilde z\|_{X_{T_R}}
 \\
 &\quad+
 \left(
 2C_BM_RT_R^{\theta_B}
 +
 C_fT_R^\alpha
 \right)
 \|u_R^z-u_R^{\widetilde z}\|_{X_{T_R}}.
\end{align*}
Using \eqref{eq:TR-contraction-condition} and absorbing the final term
gives
\begin{equation}\label{eq:localized-solution-Lipschitz}
 \|u_R^z-u_R^{\widetilde z}\|_{X_{T_R}}
 \leq
 4\|z-\widetilde z\|_{X_{T_R}}.
\end{equation}
Thus \(\mathfrak S_R\) is globally Lipschitz.

By Theorem~\ref{thm:convolution-nclt},
\[
 Z_k^N|_{[0,T_R]}
 \Longrightarrow
 Z_k|_{[0,T_R]}
 \quad\text{in }X_{T_R}.
\]
Since \(\mathfrak S_R\) is continuous, the continuous mapping theorem
gives
\[
 u_R^N
 =
 \mathfrak S_R(Z_k^N|_{[0,T_R]})
 \Longrightarrow
 \mathfrak S_R(Z_k|_{[0,T_R]})
 =
 u_R
\]
in \(X_{T_R}\).  This proves
\eqref{eq:localized-solution-convergence}.
\qed

% ============================================================
\subsection{Relation to the original equation}
% ============================================================

\begin{rmk}\label{rem:no-maximal-lifetime-convergence}
Theorem~\ref{thm:localized-solution-nclt} does not assert convergence
of random maximal lifetimes.  It gives weak convergence in the fixed
Polish space
\[
 C([0,T_R];L_\sigma^q)
\]
on a common deterministic interval.  Hence it does not require the
continuation argument discussed in Section~\ref{sec:existence}.
\end{rmk}

\begin{rmk}\label{rem:projection-inactive}
On the events
\[
 \{\|Z_k^N\|_{X_{T_R}}\leq R\}
 \quad\text{and}\quad
 \{\|Z_k\|_{X_{T_R}}\leq R\},
\]
respectively, the radial projection is inactive.  On these events,
the corresponding localized equation agrees with the original mild
equation on \([0,T_R]\).
\end{rmk}

Together,
Theorems~\ref{thm:functional-hermite-nclt},
\ref{thm:convolution-nclt}, and
\ref{thm:localized-solution-nclt}
establish the radially localized non-central limit approximation.  The
first theorem proves functional convergence of the complete
cylindrical forcing, the second transfers that convergence through the
endpoint-compensated Mittag--Leffler convolution, and the third
transfers it through a uniformly continuous nonlinear solution map on
a common deterministic interval.

\section{Conclusion and open problems}\label{sec:conclusion}

We have established a well-posedness and regularity theory for time-fractional stochastic Navier-Stokes equations driven by a cylindrical Hermite noise of arbitrary order $k\ge1$, via a Da Prato-Debussche-type decomposition in a scale of Lebesgue spaces $L^q_\sigma(D)$ on which the bilinear term is controlled by a Calder\'on-Zygmund estimate. Existence of a unique maximal solution and H\"older regularity of the stopped solution hold under the single admissibility condition $\al+2H>2$.  The moment estimates of Section~\ref{sec:moments} give an upper bound with exponent $k$ for the quadratic self-interaction of the noise, a consequence of the quadratic nonlinearity doubling the relevant Wiener chaos order.  Finally, a functional non-central limit theorem, in which the Dobrushin--Major--Taqqu scaling depends on the Hermite order, yields a radially localised approximation of the solution on a short deterministic interval.

Several questions remain open.

\begin{enumerate}

\item \textbf{Global existence.} Is $\tau^*=\infty$ almost surely for the full admissible range $\al+2H>2$? Solutions can be continued up to a maximal time $\tau^*$ by the continuation argument, with the blow-up alternative $\limsup_{t\uparrow\tau^*}\|v(t)\|_{L^q_\sigma}=\infty$. For the deterministic two-dimensional Navier--Stokes equations global existence follows from energy estimates; in the present fractional, stochastic setting the Caputo memory term and the non-semimartingale character of Hermite noise ($H\ne\frac12$) make a direct energy argument significantly harder.

\item \textbf{Moments and intermittency of the solution.} Only the moments of $\sup\|Z_k\|$ and of the self-interaction $B(Z_k)$ are controlled, with exponent $k$ for the latter. Bounding the moments of $u(t)$ at deterministic times needs global existence; a matching lower bound would pin down the exact intermittency exponent. Both are open.

\item \textbf{Multiplicative noise.} This paper treats additive noise only. Extending the results to multiplicative noise $\sigma(u)\dot Z_H^k$ would require stochastic integration with random integrands in the Hermite chaos setting, demanding Malliavin calculus estimates for the nonlinear Navier-Stokes terms that are not yet available in this generality; see \cite{CoupekDuncanPasikDuncan2022} for the Rosenblatt case ($k=2$) and \cite{CoupekKrizSvoboda2025} for partial results at general $k$.

\item \textbf{Dimension three.} The three-dimensional case is physically more relevant but technically harder: the Sobolev embedding underlying Proposition~\ref{prop:sobolev} and the Calder\'on-Zygmund estimate of Lemma~\ref{lem:bilinear} are both specific to $d=2$, and their three-dimensional analogues require a correspondingly narrower range of admissible Lebesgue exponents $q$.

\item \textbf{Rigorous parameter estimation.} The identification of $\al$, $H$, and $k$ from discrete observations of the velocity field is a natural statistical problem raised by the non-central limit theorem of Section~\ref{sec:nclt}.

\item \textbf{Scaling covariance.} Take $f=0$, $u_0=0$, neglect $A$ and use $\cS_\al(r)\approx r^{\al-1}/\Gamma(\al)$. Writing $\tilde u(t)=a\,u(ct)$, the quadratic term forces $a=c^\al$, and the self-similarity $dZ(cs)\overset d=c^HdZ(s)$ multiplies the noise by $c^{2\al+H-1}$. Setting $c=\la^{1/(2\al+H-1)}$ gives, formally,
\[
u^{(\la)}(t)\stackrel{(d)}{\approx}\la^{\frac{\al}{2\al+H-1}}\,u\!\left(\la^{\frac1{2\al+H-1}}\,t\right),
\]
where $u^{(\la)}$ is the solution with forcing $\la\,\dot Z_H^k$. A rigorous proof would require a spectral analysis of the Mittag-Leffler operators and is left for future work.

\end{enumerate}

\end{document}